\documentclass[12pt]{amsart}
\usepackage{ amsmath, amsthm, amsfonts, amssymb, color}
 \usepackage{mathrsfs}
\usepackage{amsfonts, amsmath}
 \usepackage{amsmath,amstext,amsthm,amssymb,amsxtra}
 \usepackage{txfonts}
 \usepackage[colorlinks, citecolor=blue,pagebackref,hypertexnames=false]{hyperref}
 \allowdisplaybreaks
 \usepackage{pgf,tikz}

\begin{document}

\hfuzz=6pt

\widowpenalty=10000

\newtheorem{cl}{Claim}

\newtheorem{theorem}{Theorem}[section]
\newtheorem{proposition}[theorem]{Proposition}
\newtheorem{coro}[theorem]{Corollary}
\newtheorem{lemma}[theorem]{Lemma}
\newtheorem{definition}[theorem]{Definition}
\newtheorem{example}[theorem]{Example}
\newtheorem{remark}[theorem]{Remark}
\newcommand{\ra}{\rightarrow}
\renewcommand{\theequation}
{\thesection.\arabic{equation}}
\newcommand{\ccc}{{\mathcal C}}
\newcommand{\one}{1\hspace{-4.5pt}1}

 \def \Lips  {{   \Lambda}_{L}^{ \alpha,  s }(X)}
\def\BL {{\rm BMO}_{L}(X)}
\def\HAL { H^p_{L,{at}, M}(X) }
\def\HML { H^p_{L, {mol}, M}(X) }
\def\HM{ H^p_{L, {mol}, 1}(X) }
\def\Ma { {\mathcal M} }
\def\MM { {\mathcal M}_0^{p, 2, M, \epsilon}(L) }
\def\dMM { \big({\mathcal M}_0^{p,2, M,\epsilon}(L)\big)^{\ast} }
  \def\RR {  {\mathbb R}^n}
\def\HSL { H^p_{L, S_h}(X) }
\newcommand\mcS{\mathcal{S}}
\newcommand\mcB{\mathcal{B}}
\newcommand\D{\mathcal{D}}
\newcommand\C{\mathbb{C}}
\newcommand\N{\mathbb{N}}
\newcommand\R{\mathbb{R}}
\newcommand\G{\mathbb{G}}
\newcommand\T{\mathbb{T}}
\newcommand\Z{\mathbb{Z}}

\newcommand\CC{\mathbb{C}}
\newcommand\NN{\mathbb{N}}
\newcommand\ZZ{\mathbb{Z}}

\renewcommand\Re{\operatorname{Re}}
\renewcommand\Im{\operatorname{Im}}

\newcommand{\mc}{\mathcal}

\def\SL{\sqrt[m] L}
\newcommand{\la}{\lambda}
\def \l {\lambda}
\newcommand{\eps}{\varepsilon}
\newcommand{\pl}{\partial}
\newcommand{\supp}{{\rm supp}{\hspace{.05cm}}}
\newcommand{\x}{\times}

\newcommand\wrt{\,{\rm d}}

\title[Sharp spectral multipliers  ]{%
Sharp spectral multipliers for operators satisfying generalized  Gaussian  
estimates}
\author{Adam Sikora, \ Lixin Yan \ and Xiaohua Yao}

\address {Adam Sikora, Department of Mathematics, Macquarie University, NSW 2109, Australia}
\email{sikora@maths.mq.edu.au}
\address{Lixin Yan, Department of Mathematics, Sun Yat-sen (Zhongshan) University,
Guangzhou, 510275, P.R. China}
\email{mcsylx@mail.sysu.edu.cn}
\address{Xiaohua Yao, Department of Mathematics, Central China Normal University, Wuhan, 430079, P.R. China}
\email{yaoxiaohua@mail.ccnu.edu.cn }
\date{\today}
\subjclass[2000]{42B15, 42B20,   47F05.}
\keywords{ Spectral multipliers, restriction type condition,
non-negative self-adjoint  operator,
  heat semigroup,  $m$-th order generalized Gaussian estimates, space of homogeneous type.}

\begin{abstract}
Let $L$ be a non-negative self adjoint operator acting on $L^2(X)$
where $X$ is a space of homogeneous type. Assume that $L$ generates
a holomorphic semigroup $e^{-tL}$ whose kernels $p_t(x,y)$
satisfy  generalized $m$-th order Gaussian estimates.  In this article,
 we study singular and dyadically supported spectral multipliers for abstract
self-adjoint operators.   We show that in this setting sharp spectral
multiplier  results follow from Plancherel or Stein-Tomas type estimates.
These results are applicable to spectral multipliers for large classes of operators
 including  $m$-th order  elliptic differential operators with constant coefficients,
biharmonic operators with rough potentials and Laplace type operators acting on fractals.
\end{abstract}

\maketitle

 \tableofcontents

\section{Introduction}
\setcounter{equation}{0}

 This paper is devoted to the theory of spectral multipliers of self-adjoint differential type operators.
 This is a classical area of
harmonic analysis, which  has attracted a lot of attention during in the last  fifty years or so.
The literature devoted to the subject is so broad that it is impossible to provide complete and comprehensive
bibliography. Therefore we quote only a few  papers, which are directly related to our study and refer readers to
\cite{Blu, COSY, C3, ChS, CDMY, CowS, DM, DOS, F1, F2, Heb, Ho2, Ho3, Ho4, KU, S3, SS, Sog1, St2, T1, Tom}
and the references within for further relevant literature.

We consider a measure space $X$ and  a non-negative
self-adjoint  operator $L$ acting on~$L^2(X)$. Such an operator admits a spectral
resolution  $E_L(\lambda)$  and for  any  bounded  Borel function $F\colon [0, \infty)
\to {\mathbb C}$, one can define the operator $F(L)$
\begin{equation}\label{e1.1}
F(L)=\int_0^{\infty}F(\lambda) \wrt E_L(\lambda).
\end{equation}
By spectral theory the operator $F(L)$ is bounded on $L^2(X)$ and its norm is equal to $L$-essential
supremum norm of $F$. Spectral multiplier  theorems investigate under what conditions on function
$F$ the operator $F(L)$ can be extended to a bounded operator acting on Lebesgue spaces $L^p(X)$ for
 some range of $p$. Usually one looks for condition formulated in terms of differentiability of function $F$.
Spectral multiplier theorems are closely related to the problem of Bochner-Riesz sumability.
There are some subtle differences between two problems but the essential core  of Bochner-Riesz sumability
problem and the spectral multiplier  theorems is identical.

We would like to mention  three different aspects of spectral multipliers theory and Bochner-Riesz analysis.

\smallskip

$\bullet$ Dyadically  supported spectral multipliers. Here one assumes that a function $F\in C_c(a,b)$  for some
$0<a<b$ is compactly  supported and one tries to  find necessary conditions to ensure that
$$
\sup_{t>0}\|F(tL)\|_{p \to p}\le C < \infty
$$
for some $p\in [1, \infty]$  or range of such $p$, where $F(tL)$ is defined by the spectral resolution.
Usually the condition of $F$ is expressed in terms of Sobolev spaces $W^s_q(\R)$  and constant $C$ is proportional
to the corresponding Sobolev norm of function $F$. Compact support assumption could be misleading here
as it can be essentially weakened but it is convenient because  the dyadic decomposition trick is often used in the theory
of spectral multipliers. Usually the proof of sharp results of this type requires
Plancherel or restriction type estimates, which we discuss below, see Section~\ref{sec4}.

\smallskip

$\bullet$ Singular integral spectral multipliers, see Sections \ref{sec3} and \ref{sec5}. One considers auxiliary nonzero compactly supported
function $\eta \in C_c(a,b)$. Then for some Sobolev space $W^s_q(\R)$ one can define a ``local Sobolev norm'' by the formula
$$
\|F\|_{LW^s_{q}}= \sup_{t>0} \|\eta  (\cdot ) \, F(t\cdot)  \|_{W_q^s(\R)},
$$
where $\|F\|_{W_q^s(\R)}=\| (1-{\frac{d^2}{d x^2}})^{s/2} F\|_q$.
Up to a constant this definition does not depend  on a choice of the auxiliary function $\eta$ as long as it is a non zero function.
In singular spectral multipliers one wants to obtain  estimates of $L^p \to L^p $ or weak $(p,p)$ norm of the operator
$F(L)$ in terms of the norm $\|F\|_{LW^s_{q}}$. In principle one  expects  that compact spectral multipliers
imply singular spectral multipliers even though   one can construct examples where
$\sup_{t>0}\|\eta(L) F(tL)\|_{p \to p}\le C < \infty $ but $F(L)$ is unbounded, see well-known counterexample
of Littman, McCarthy and Rivi\`ere \cite{LMR}.

\smallskip

$\bullet$ The most essential point of spectral multiplier theorems and Bochner Riesz analysis is an investigation of
 of Plancherel or restriction type estimates. Such estimates  are essentially required to  obtain compactly supported
 spectral multiplier. They originate from classical Fourier analysis where one considers  so-called restriction problem: {\em
 describe all pairs  of exponent $(p,q)$ such that the restriction operator
 $$
 R_\lambda(f)(\omega) =\hat{f}(\lambda \omega)
 $$
 is bounded from $L^p(\R^n) \to L^q({\bf S}^{n-1})$.} Here $\hat{f}$ is the Fourier transform of $f$ and
 $\omega\in   {\bf S}^{n-1}$ is a point on the unit sphere.  It is not difficult to notice
 that if $E_{\sqrt{\Delta}}$ is the spectral resolution for the standard Laplace operator,
 then
 $$
d E_{\sqrt{\Delta}}(\lambda) =\frac{\lambda^{n-1}}{(2\pi)^n} R_\lambda^*R_\lambda
 $$
(compare \cite{GHS}).  Now if one knows (as it is the case in Stein-Tomas restriction estimates, see \cite{Tom})
that the restriction operator is bounded for some  pair $(p,2)$, then by $T^*T$ trick it
follows that  $\frac{d}{d\lambda} E_{\sqrt{\Delta}}(\lambda)$ extends to a bounded operator acting form
space $L^p(\R^n) \to L^{p'}(\R^n)$, where $p'$ is conjugate exponent of $p$ ($1/p+1/p'=1$).
Note that these  estimates can be expressed purely in terms of spectral resolutions of self-adjoint operator $L$.
Motivated by the example of the standard Laplace operator we introduce below condition ${\rm (ST^{q }_{p, 2, m})}$.
 One of most significant part of spectral multiplier results and Bochner-Riesz analysis is to prove
estimates of this type for some class of differential operators. Remarkable results of this type were
obtained in \cite{GHS}. Some other  examples are described in \cite{COSY}. In the case when $p=1$ the problem
quite often simplifies because of existence of underlying Plancherel measure and for example for homogeneous
operators efficient  restriction estimates  type results
are automatically true, see \cite{DOS}. To illustrate our abstract spectral multiplier results, we describe some well known examples of
operators which satisfy ${\rm (ST^{q }_{p, 2, m})} $ in Section \ref{appl5}. In Section \ref{appl5.3} we describe
 new restriction estimates of this type.

\smallskip

In this paper we are mainly focus on proving that appropriate restriction type estimates imply
sharp compactly supported spectral multiplier results and that   the  singular integral version follows
from compactly supported spectral multipliers for abstract self-adjoint operators for which the corresponding
heat kernels satisfy $m$-th order Gaussian bounds. We also discuss operators, which satisfy generalized Gaussian
estimates in the sense of  Blunck and Kunstmann, see e.g. \cite{BK3}. The case $p=1$ for such theory was
 comprehensively discussed in \cite{DOS}. However if $p\neq 1$ the problem requires essentially new approach.

Under assumption that $L$ satisfies finite speed propagation property, see \cite{CouS, S2}, similar results
 were considered in \cite{COSY}. However there are many interesting examples  of whole significant classes of
operators which do not satisfy finite speed propagation property for the wave equation but satisfy $m$-th order
Gaussian bound. For example $m$-th order differential operators or Laplace like operators defined on fractals,
see for example \cite{BP, Str1, Str2}. The results obtained in this paper can be applied to these operators.
Finite speed propagation property is equivalent to the  second order Gaussian bounds, see \cite{CouS, S2},
so our paper can be regarded as the generalization of \cite{COSY}. In particular our results apply to all examples discussed
there. Note however that we are not able to obtain endpoint results in the current setting.

Our proof  that compactly (dyadically) supported spectral multipliers imply singular integral multipliers is inspired by the work of Seeger
and Sogge \cite{S3, SS}. However,
  there is no assumption on the regularity
in variables $x$ and $y$ on the kernels $p_t(x,y)$ of the semigroup $e^{-tL}$,
thus techniques of Calder\'on--Zygmund
theory (\cite{S3, SS}) are not applicable. The lacking of smoothness of the kernel was indeed the main
obstacle   and it was overcome by
using   an approach to singular integral theory initiated by \cite{Heb}
and developed in \cite{DM, ACDH}.
 In this approach to obtain additional cancellation instead of subtracting some average of a function
 one subtracts appropriate multiplier of the operator $L$ applied to the considered function.

 In the sequel we always assume that considered ambient space is  a metric measure space
  $(X,d,\mu)$ with metric $d$ and Borel measure $\mu$.
We denote by
$B(x,r)=\{y\in X,\, {d}(x,y)< r\}$  the open ball
with centre $x\in X$ and radius $r>0$. We also  assume that the space $X$ is homogeneous
that is  it satisfies the doubling condition. It allows us to consider the homogeneous dimension of the
space $X$. To be more precise we put  $V(x,r)=\mu(B(x,r))$ the volume of $B(x,r)$ and we say that
$(X, d, \mu)$ satisfies the doubling property (see Chapter 3, \cite{CW})
if there  exists a constant $C>0$ such that \begin{eqnarray}
V(x,2r)\leq C V(x, r)\quad \forall\,r>0,\,x\in X. \label{e1.2}
\end{eqnarray}
If this is the case, there exist  $C, n$ such that for all $\lambda\geq 1$ and $x\in X$
\begin{equation}
V(x, \lambda r)\leq C\lambda^n V(x,r). \label{e1.3}
\end{equation}
In the sequel we want to consider $n$ as small as possible.
Note that in general one cannot take infimum over such exponents $n$ in \eqref{e1.3}.
In the Euclidean space with Lebesgue measure, $n$ corresponds to
the dimension of the space. In our results critical index is always expressed in terms of
homogeneous dimension $n$. Usually existence of
$s>n(1/2-1/p)$ derivatives of function $F$ is a sharp optimal condition in most of spectral
multiplier results, both compact and singular. However there is a subtle but of huge significance difference
 between existence of this number of derivatives in $L^2(\R)$ versus $L^\infty(\R)$.
 Improvement of the results from $L^\infty$ to
$L^2$ always requires some form of restriction or Plancherel type of estimates.
This  $L^2$ spectral multiplier results essentially corresponds to calculation of
critical exponent in  Bochner-Riesz means analysis and is related to  Bochner-Riesz
conjecture. Obtaining sharp results in this context is regarded as one of most crucial tasks
in harmonic analysis.

\section{Preliminaries}
\setcounter{equation}{0}

We commence with describing  our notation and basic assumptions.  We   often just use $B$ instead of $B(x, r)$.
Given $\lambda>0$, we write $\lambda B$ for the $\lambda$-dilated ball
which is the ball with the same centre as $B$ and radius $\lambda r$.
For $1\le p\le+\infty$, we denote the
norm of a function $f\in L^p(X)$ by $\|f\|_p$, by $\langle .,. \rangle$
the scalar product of $L^2(X)$, and if $T$ is a bounded linear operator from $
L^p(X)$ to $L^q(X)$, $1\le p, \, q\le+\infty$, we write $\|T\|_{p\to q} $ for
the  operator norm of $T$.
Given a  subset $E\subseteq X$, we  denote by  $\chi_E$   the characteristic
function of   $E$ and  set
$$
P_Ef(x)=\chi_E(x) f(x).
$$
For every  $B=B(x_B, r_B)$,  set $A(x_B, r_B, 0)=B$ and
$$
A(x_B, r_B, j)=B(x_B, (j+1)r_B)\backslash B(x_B, jr_B), \ \ \ \ \ j=1, 2, \ldots
$$
For a given  function $F: {\mathbb R}\to {\mathbb C}$ and $R>0$, we define the
function
$\delta_RF:  {\mathbb R}\to {\mathbb C}$ by putting
 $\delta_RF(x)= F(Rx).$  Given $p\in [1, \infty]$, the conjugate exponent $p'$
is defined by $1/p +1/p' =1.$
We will also use the
Hardy-Littlewood maximal operator  ${\mathcal M}f$ which is defined by
$$
{\mathcal M}f(x)=\sup_{B\ni x}{\frac{1}{V(B)}}\int_{B}
|f(y)|d\mu(y),
$$
where the {\it sup} is taken over all balls $B$ containing $x$.

 \medskip

\noindent
{\bf 2.1.\, Generalized Gaussian   estimates and Davies-Gaffney estimates.}
We now described the notion of the Generalized Gaussian estimates introduced by Blunck and Kunstmann,
see \cite{BK1, BK2, BK3}.
  Consider a non-negative self-adjoint operator $L$  and numbers
  $m\geq 2$ and $ 1\leq p\leq 2\leq q\leq \infty$ with $p<q$.
  We say that   the semigroup generated by $L$, $e^{-tL}$  satisfies
{\it generalized Gaussian  $(p,q)$-estimates  of order $m$},
if there exist constants $C, c>0$ such that
for all $t>0$,  and all $x,y\in X,$
$$\leqno{\rm(GGE_{p, q, m})} \hspace{1cm}
\big\|P_{B(x, t^{1/m})} e^{-tL} P_{B(y, t^{1/m})}\big\|_{p\to {q}}\leq
C V(x,t^{1/m})^{-({\frac{1}{ p}}-{1\over q})} \exp\Big(-c\Big({d(x,y) \over    t^{1/m}}\Big)^{m\over m-1}\Big).
$$

\medskip

Note that condition ${\rm(GGE_{p, q, m})}$ for the special case $(p, q)=(1, \infty)$ is equivalent to
$m$-th order   Gaussian estimates (see   Proposition 2.9, \cite{BK1}). This means that the
semigroup $e^{-tL}$ has integral kernels $p_t(x,y)$ satisfying the following estimates
$$\leqno{\rm (GE_m)} \hspace{1cm}
 |p_t(x,y)|\leq {C\over V(x,t^{1/m})} \exp\Big(-c\Big({d^m(x,y) \over    t}\Big)^{1\over m-1}\Big) \ \
{\rm for \ } x,y\in X, \ t>0.
$$
There are numbers of operators which satisfy generalized Gaussian estimates and, among them,
 there exist many for which classical Gaussian estimates ${\rm (GE_m)}$ fail.
  This happens, e.g., for Schr\"odinger operators with rough
 potentials \cite{ScV}, second order elliptic operators with rough  lower order terms \cite{LSV}, or
 higher order elliptic operators with bounded measurable coefficients
 \cite{D2}.

The following result originally  stated  in  \cite[Lemma 2.6]{U} (see also \cite[Theorem 2.1]{BL}) shows that generalized Gaussian estimates can be extended from
real times $t>0$ to complex times   $z\in {\mathbb C}$ with ${\rm Re} z>0$.

\begin{lemma}\label{le2.1} Let $m\geq 2$ and $ 1\leq p\leq 2\leq q\leq \infty$, and $L$ be a
 non-negative self-adjoint operator   on $L^2(X)$. Assume that  there exist constants $C, c>0$ such that
for all $t>0$,  and all $x,y\in X,$
\begin{eqnarray*}
\big\|P_{B(x, t^{1/m})} e^{-tL} P_{B(y, t^{1/m})}\big\|_{p\to {q}}\leq
 C V(x,t^{1/m})^{-({1\over p}-{1\over q})} \exp\Big(-c\Big({d(x,y) \over    t^{1/m}}\Big)^{m\over m-1}\Big).
\end{eqnarray*}
Let $r_z=({\rm Re}\, z)^{{1\over m}-1 } |z|$
for each  $z\in{\mathbb C}$ with ${\rm Re}\, z >0$.
\begin{itemize}

\item[(i)]  There exist two positive constants $C' $ and $ c'$ such that for all $r>0, x\in X, $
 and $z\in{\mathbb C}$ with ${\rm Re}\, z >0$
\begin{eqnarray*}
 &&\hspace{-1.5cm}\big\|P_{B(x, r)}    e^{-zL} P_{B(y, r)}\big\|_{p \to q}\\
 &\leq &C' V(x,r)^{-({1\over p }-{1\over q})} \Big(1+{r\over r_z}\Big)^{n({1\over p }-{1\over q})}
  \Big({|z|\over {\rm Re}\,  z}\Big)^{n({1\over p }-{1\over q})}
  \exp\Big(-c' \Big({ d(x,y)   \over   r_z}  \Big)^{m\over m-1} \Big).
\end{eqnarray*}

\item[(ii)]  There exist two positive constants $C''$ and $ c''$ such that for all $r>0, x\in X, k\in{\mathbb N}$
 and $z\in{\mathbb C}$ with ${\rm Re}\, z >0$
\begin{eqnarray*}
 &&\hspace{-1.5cm}\big\|P_{B(x, r)}    e^{-zL} P_{A(x, r, k)}\big\|_{p \to q}\\
 &\leq &C'' V(x,r)^{-({1\over p }-{1\over q})} \Big(1+{r\over r_z}\Big)^{n({1\over p}-{1\over q})}
  \Big({|z|\over {\rm Re} \, z}\Big)^{n({1\over p }-{1\over q})} k^n
  \exp\Big(-c'' \Big({ r   \over    r_z} k\Big)^{m\over m-1} \Big).
\end{eqnarray*}
\end{itemize}
 \end{lemma}
 \begin{proof}
 For the 
 detailed proof
 we refer  readers to \cite{U}.
  Here we only  want to mention that the proof of Lemma \ref{le2.1} relies on the Phragm\'en-Lindel\"of theorem.
 \end{proof}

 \medskip
  Next suppose that $m\geq 2$.
  We say that the semigroup $e^{-tL}$ generated by non-negative self-adjoint operator $L$ satisfies
 {\it $m$-th order Davies-Gaffney  estimates},
if there exist constants $C, c>0$ such that
for all $t>0$,  and all $x,y\in X,$
$$\leqno{\rm(DG_{m})} \hspace{1cm}
\big\|P_{B(x, t^{1/m})} e^{-tL} P_{B(y, t^{1/m})}\big\|_{2\to {2}}\leq
C   \exp\Big(-c\Big({d(x,y) \over    t^{1/m}}\Big)^{m\over m-1}\Big).
$$

 \medskip

Note that if  condition ${\rm (GGE_{p, q, m})}$ holds for
 for some $1\leq p\leq 2\leq q\leq \infty$ with $p<q$,
then  the semigroup
$e^{-tL}$  satisfies estimate ${\rm (DG_m)}$.

The following lemma describes a useful consequence of $m$-order
Davies-Gaffney estimates.

 \medskip

\begin{lemma}\label{le2.2}
Let $m\geq 2$ and  let  $e^{-tL}$ be a semigroup generated by a non-negative, self-adjoint
operator   $L$ acting on $L^2(X)$ satisfying Davies-Gaffney  estimates ${\rm (DG_m)}$.
Then for every  $M>0$, there exists a constant $C=C(M)$ such that
for every $j=2,3,\ldots$

\begin{eqnarray}\label{e2.1} \hspace{1cm}
 \big\|P_{B}    F(\sqrt[m]{L}) P_{A(x_B, r_B, j)}\big\|_{2\to 2}\leq
 C j^{-M}  \big(R r_B)^{-(M+n)}   \|\delta_R F\|_{W^{M+n+1}_2}
\end{eqnarray}
for all  balls $B\subseteq  X$, and all Borel functions $F$  such that supp $F\subseteq [{R/4}, R]$.
 \end{lemma}

\begin{proof} Let $G(\lambda)=\delta_RF(\sqrt[m]{\lambda}) e^{\lambda}.$ In virtue of the Fourier inversion formula
$$
G(L/R^m)e^{-L/R^m}={1\over 2\pi} \int_{\mathbb R} e^{(i\tau-1)R^{-m}L} {\hat G}(\tau)d\tau
$$
so
$$
 \|P_{B}    F(\sqrt[m]{L}) P_{A(x_B, r_B, j)}\|_{2\to 2} \leq
{1\over 2\pi} \int_{\R} |{\hat G}  (\tau)| \,   \|P_{B}   e^{(i\tau-1)R^{-m}L} P_{A(x_B, r_B, j)} \|_{2\to 2}
  d\tau.
$$
  By (ii) of Lemma~\ref{le2.1} (with $r_z =\sqrt{1+\tau^2}/R$),
 \begin{eqnarray*}
\|P_{B}   e^{(i\tau-1)R^{-m}L} P_{A(x_B, r_B, j)} \|_{2\to 2}&\leq&
Cj^n
  \exp\bigg(-c\Big({ Rj r_B \over    \sqrt{1+\tau^2}}\Big)^{m\over m-1}   \bigg)\\
  &\leq& C_M  j^n   \Big( {  R jr_B \over \sqrt{1+\tau^2}    } \Big)^{-M-n}\\
  &\leq& C j^{-M} \big(1+\tau^2)^{ M+n \over 2}  \big(R r_B)^{-(M+n)}.
\end{eqnarray*}
Therefore (compare \cite[(4.4)]{DOS})
  \begin{eqnarray*}
   &&\hspace{-1cm}\|P_{B}    F(\sqrt[m]{L}) P_{A(x_B, r_B, j)}\|_{2\to 2}\\
   &\leq&
 Cj^{-M}  \big(R r_B)^{-(M+n)}
   \int_{\R} |{\hat G}  (\tau)|  \big(1+\tau^2)^{ M+n \over 2}  d\tau\nonumber\\
   &\leq&
 C j^{-M}  \big(R r_B)^{-(M+n)}
   \Big(\int_{\R} |{\hat G}  (\tau)|^2  \big(1+\tau^2)^{M+n+1}   d\tau\Big)^{1/2}
    \Big(\int_{\R} \big(1+\tau^2)^{-1}  d\tau\Big)^{1/2}\nonumber\\
  &\leq&    C  j^{-M}  \big(R r_B)^{-(M+n)}   \|G\|_{W^{M+n+1}_2}.
\end{eqnarray*}
 However, supp $F\subseteq [R/4, R]$ and supp $\delta_R F\in [1/4, 1]$ so
 $$
  \|G\|_{W^{M+n+1}_2} \leq C \|\delta_R F\|_{W^{M+n+1}_2}.
 $$
This ends the proof of Lemma \ref{e2.1}.
\end{proof}

\medskip

\noindent
{\bf 2.2.\, The Stein-Tomas restriction type condition.}
Consider a non-negative self-adjoint operator $L$  and numbers
$p $ and $q$ such that $1\leq p< 2$ and $1\leq
q\leq\infty$. We say that  $L$ satisfies the
    {\it  Stein-Tomas restriction type condition}   if:
  for any $R>0$ and all Borel functions $F$ such that $\supp F \subset [0, R]$,
$$
\big\|F(\SL)P_{B(x, r)} \big\|_{p\to 2} \leq CV(x,
r)^{{1\over 2}-{1\over p}} \big( Rr \big)^{n({1\over p}-{1\over
2})}\big\|\delta_RF\big\|_{q}
\leqno{\rm (ST^{q}_{p, 2, m})}
$$
  for all
  $x\in X$ and all $r\geq 1/R$.

\medskip

An interesting  alternative approach to restriction type estimates is
investigated by
Kunnstman and Uhm in \cite{KU, U}, see (4.3) of \cite{KU} and
the Plancharel
condition (5.30) of \cite{U}. Let us point out that  estimate ${\rm (ST^{2}_{p, 2, m})}$
implies sharp Bochner Riesz results for all $p$.

 Note that   if condition ${\rm (ST^{q }_{p, 2, m})} $ holds for  some $q\in [1, \infty)$,
 then ${\rm (ST^{\tilde{q}}_{p, 2, m})} $  holds for all $\tilde{q}\geq q$ including the case
 $\tilde{q}=\infty$.

The next lemma shows that if ${q}=\infty$ then condition  ${\rm (ST^{{\infty}}_{p, 2, m})} $
follows form the standard elliptic estimates.

\medskip

\begin{proposition}\label {prop2.3} Suppose that $(X, d, \mu)$ satisfies
 property  \eqref{e1.2} and \eqref{e1.3}.  Let  $1\leq p<2 $ and
$N>n(1/p-1/2)$. Then   ${\rm (ST^{\infty}_{p, 2, m})}$   is equivalent
to each of the following conditions:
\begin{itemize}
\item[(a)]  For all $x\in X$ and  $r\geq t>0$    
$$
\big\|e^{-t^mL}P_{B(x, r)}\big\|_{p\to 2} \leq
CV(x, r)^{{1\over 2}-{1\over p}} \Big({r\over  {t}}\Big)^{n({1\over p}-{1\over 2})}.
\leqno{\rm (G_{p,2,m})}
$$

\item[(b)]
For all $x\in X$ and    $r\geq  t >0$   
$$
\big\|(I+t\SL)^{-N}P_{B(x, r)}\big\|_{p\to 2} \leq
CV(x, r)^{{1\over 2}-{1\over p}} \left({r\over  {t}}\right)^{n({1\over p}-{1\over 2})}.
\leqno{\rm (E_{p, 2,m})}
$$
\end{itemize}
\end{proposition}

\begin{proof} The proof is originally given in \cite{COSY} only second-order operators. However,
with some minor modifications, the proof can be adapted to the $m$th-order version,  and we omit the detail here.
\end{proof}

\medskip

\medskip
The following lemma is a standard known result in the theory of spectral multipliers of non-negative selfadjoint
operators and it is a version of \cite[Lemma 4.5]{COSY} adjusted to the setting of $m$-order operators so we
use the same notation.

\begin{lemma}\label{le2.4}
Suppose  that
 operator $L$  is a non-negative self-adjoint operator $L$ on $L^2(X)$ satisfying
Davies-Gaffney estimates ${\rm (DG_m)}$  and condition ${\rm (G_{p_0,2, m})}$ for some
  $1\leq p_0< 2$.
\begin{itemize}
\item[(a)]
Assume in addition that $F$ is  an even bounded Borel function  such that
$$
\sup_{t>0}\|\eta\delta_tF\|_{C^k}<\infty
$$
for some integer $k> n/2 + 1$ and some non-trivial function $\eta\in C_c^{\infty}(0, \infty)$.
Then the operator  $F(\SL)$ is bounded on $L^{p}(X)$ for all $p_0<p<p_0'$. \\

\item[(b)] Assume in addition that $\psi $  be  a function in $
{\mathscr S} ({\mathbb{R}})$ such that $\psi(0)=0$. Define the
quadratic functional for $f\in L^2(X)$
\begin{eqnarray*}
{\mathcal G}_L(f)(x)=\Big( \sum_{j\in{\mathbb Z}}  |\psi(2^j\SL)f  |^2\Big)^{1/2}.
\end{eqnarray*}
Then ${\mathcal G}_L$ is bounded on $L^{p}(X)$ for all $p_0<p<p_0'.$
\end{itemize}
\end{lemma}

\begin{proof}  It follows from  ${\rm (G_{p_0,2, m})}$ that
\begin{eqnarray}\label{e2.2}
  \| P_{ B(x, t^{1/m})} e^{-tL} P_{B(y,t^{1/m})} \|_{p_0\to 2} \le C V(x,t^{1/m})^{\frac{1}{2} - \frac{1}{p_0}}.
\end{eqnarray}
Let $r\in (p_0, 2)$. By   \eqref{e2.2} and Davies-Gaffney estimates ${\rm (DG_m)}$,
the Riesz-Thorin interpolation theorem gives   the following $L^r-L^2$ off-diagonal estimate
\begin{eqnarray}\label{e2.3}\hspace{1cm}
 \| P_{ B(x, t^{1/m})} e^{-tL} P_{B(y,t^{1/m})}
\|_{r\to 2} \le C V(x,t^{1/m})^{\frac{1}{2} - \frac{1}{r}}   \exp\Big(-c\Big({d(x,y) \over    t^{1/m}}\Big)^{m\over m-1}\Big)
\end{eqnarray}
for all $x,y\in X$ and $t>0$.
Assertion (a) then follows from \cite{Blu}.  The latter  off-diagonal  estimate implies that $L$ has a bounded holomorphic
functional calculus on $L^p$ for $p_0 < p < p_0'$ (see \cite{Blu}).  It is known that the holomorphic functional calculus  implies
 the quadratic estimate of  assertion  (b) (see \cite{CDMY, Mc}).
 \end{proof}

\section{A criterion for $L^p$ boundedness of spectral multipliers}\label{sec3}
\setcounter{equation}{0}

In this section, we shall state and prove  a criterion for $L^p$ boundedness of spectral multipliers. In many cases, this  theorem allows us to
reduce the proof of the $L^p$-boundedness
of general multiplier operator $F(L)$ to obtaining estimates for operators corresponding to dyadically supported  functions. 
Then in the next section, we will show that sharp results for spectral multipliers with dyadic support follows from 
 restriction type conditions.

In what follows,  we fix a non-zero $C^{\infty}$ bump function   on $\mathbb R$ such that
\begin{eqnarray}
{\rm supp}  \phi \subseteq ({1\over 2}, 2) \ \ {\rm and} \ \
\sum_{\ell\in {\Bbb Z}}\phi(2^{-\ell}\lambda)=1 \ \ \ {\rm for\ all}\  \lambda>0
\label{e3.1}
\end{eqnarray}
and set $\phi_{\ell}(\lambda)=\phi(\lambda/2^{\ell})$.

 The aim of this section is to prove the following result.

\begin{theorem}\label{th3.1}
Let  $L$  be  a non-negative self-adjoint operator $L$ on $L^2(X)$ satisfying
Davies-Gaffney estimates ${\rm (DG_m)}$  and condition ${\rm (G_{p_0,2,m})}$ for some
  $1\leq p_0< 2$.
 Let  $F$ be  a bounded Borel function  such that for $p\in (p_0, p'_0),$
\begin{eqnarray}\label{e3.2}
  \sup_{t>0}\|\big(\phi \delta_t F\big)(\sqrt[m]{L}) \|_{p\to p} +
  \sup_{t>0}\|\big(\phi \delta_t F\big)(\sqrt[m]{L}) \|_{2\to 2} \leq A
\end{eqnarray}
holds. Then  for every $M>n/2+1$,  there exists a constant $C>0$ such that
\begin{eqnarray*}
 \|  F(\sqrt[m]{L})  \|_{p\to p}\leq CA \bigg\{ {\rm log}
 \Big(2+ { \sup_{t>0}\|\phi  \delta_t F\|_{  W^{M+n+1}_2} \over A}\Big)
 \bigg\}^{|{1\over p}-{1\over 2}|}.
\end{eqnarray*}

\end{theorem}

  \medskip

 The proof of Theorem~\ref{th3.1} is inspired by ideas developed in \cite{ACDH, DM, Heb, S3, SS}.

Let us
  introduce some tools needed in the proof. Let  $T$ be  a sublinear operator which is bounded on $L^{2}(X)$.
Let $\{A_r\}_{r>0}$ be a family of linear  operators acting on  $L^{2}(X)$.    For $f\in L^2(X)$,
we follow  \cite{ACDH} to  define
$$
{\mathcal M}^{\#}_{T, A}f(x)= \sup_{B\ni x} \Big({1\over V(B)}\int_{B} \big|T(I-A_{r_B}) f\big|^2 d\mu\Big)^{1/2},
$$
where the supremum is taken over all balls $B$ in $X$ containing $x$, and $r_B$ is the radius of $B.$

\smallskip

\begin{proposition}\label{prop3.2}
Suppose that   $T$ is  a sublinear operator which is bounded on $L^{2}(X)$ and that $q\in (2, \infty].$
Assume that $\{A_r\}_{r>0}$ is  a family of linear  operators acting on  $L^{2}(X)$ and that
\begin{eqnarray}\label{e3.3}
\Big({1\over V(B)} \int_B|T A_{r_B}f(y)|^{q} d\mu(y)\Big)^{1/q}
\leq C \big({\mathcal M} \big( | Tf |^{2}\big)^{1/2}(x)
\end{eqnarray}

\noindent
for all $f \in L^{2}(X) $, all $x\in X$ and all balls $B\ni x$, $r_B$ being
the radius  of $B$.

Then for $0<p<q$, there exists $C_p$ such that
\begin{eqnarray}\label{e3.4}
 \big\| \big({\mathcal M}\big( | Tf|^2 \big)\big)^{1/2}\|_p\leq C_p \big(  \|{\mathcal M}^{\#}_{T, A}f \big)\|_p +\|f\|_p\big)
\end{eqnarray}
 for every $f\in L^2(X)$ for which the left-hand side is finite (if $\mu(X)=\infty$, the term $C_p\|f\|_p$
 can be omitted in the right-hand side of \eqref{e3.4}.
\end{proposition}

\begin{proof} For the proof of Proposition~\ref{prop3.2}, we refer readers  to  \cite[Lemma 2.3]{ACDH}.
\end{proof}

 \bigskip

 \noindent
{\em Proof of}   { Theorem~\ref{th3.1}.}
 Given   a bounded Borel function  $F$, we consider an operator $T_F$, given by
$$
T_Ff(x)= \Big\{\sum_{k\in{\mathbb Z}} |(\phi^2_kF)(\sqrt[m]{L})f(x)|^2\Big\}^{1/2}.
$$
For this operator $T_F$, condition \eqref{e3.3} always holds for every  $2<q<p'_0$
and $A_{r_B}=I-(I-e^{-r_B^mL})^K$ for every  $K\in{\mathbb N}$.
Indeed, in virtue  of
the formula
\begin{eqnarray*}
 I-(I-e^{-r_B^mL})^K =\sum_{s=1}^K
 \left(
 \begin{array}{lcr}
K\\
 s
 \end{array}
 \right)(-1)^{s+1} e^{-sr_B^mL}
\end{eqnarray*}
and  the commutativity property
$(\phi^2_kF)(\sqrt[m]{L})e^{-sr_B^mL} =e^{-sr_B^mL}(\phi^2_kF)(\sqrt[m]{L})$,  it is enough to show that
for all $B\ni x,$
 \begin{eqnarray} \label{e3.5} \hspace{1cm}
\Big( {1\over V(B)}\int_B \big(\sum_{k\in{\mathbb Z}}
\big| e^{-sr_B^mL} (\phi^2_kF)(\sqrt[m]{L}) f(y)\big|^2\big)^{q/2} d\mu(y)\Big)^{1/q}
\leq C \big({\mathcal M} \big(  |T_Ff |^{2}\big)^{1/2}(x).
\end{eqnarray}
To prove \eqref{e3.5}, we  observe  that hypothesis  ${\rm (DG_m)}$  and  ${\rm (G_{p_0,2,m})}$
  imply ${\rm (GGE_{q',2, m})}$. By duality, ${\rm (GGE_{2,q, m})}$ holds.
By
  Minkowski's inequality, (ii) of Lemma~\ref{le2.1},
  conditions \eqref{e1.2} and \eqref{e1.3} for every
  $s=1,2, \ldots, K$ and every $B\ni x,$
  the left hand side of    \eqref{e3.5} is less than
\begin{eqnarray*}
&&\hspace{-1cm}V(B)^{-1/q}  \sum_{j=0}^{\infty}   \Big\{ \sum_{k\in{\mathbb Z}} \big(
\|P_{B}e^{-sr_B^mL} P_{A(x_B, r_B, j)}(\phi^2_kF)(\sqrt[m]{L})f \|_{q} \big)^{2}\Big\}^{1/2}\nonumber \\
&\leq&V(B)^{-1/q} \sum_{j=0}^{\infty}   \|P_{B}e^{-sr_B^mL} P_{A(x_B, r_B, j)}\|_{2\to q}
\Big\{ \sum_{k\in{\mathbb Z}}
 \|(\phi^2_kF)(\sqrt[m]{L})f \|_{L^2(A(x_B, r_B, j))}^{2}\Big\}^{1/2}\nonumber \\
&\leq&  C \sum_{j=0}^{\infty} \Big({V((j+1)B) \over  V(B)}\Big)^{1/2} e^{-c_s j^{m/(m-1)}}
\Big\{  {1\over V((j+1)B)} \int_{(j+1)B}
\sum_{k\in{\mathbb Z}} \big| (\phi^2_kF)(\sqrt[m]{L})f(y)\big|^2 d\mu(y) \Big\}^{1/2}\nonumber \\
&\leq& C  \sum_{j=0}^{\infty} e^{-c_sj^{m/(m-1)}} (1+j)^{n/2}\big({\mathcal M} \big( | T_Ff |^{2}\big)^{1/2}(x) \nonumber \\
 &\leq& C  \big({\mathcal M} \big( | T_Ff |^{2}\big)^{1/2}(x). \nonumber
\end{eqnarray*}
The above estimates yield   \eqref{e3.5}.

Define,  for every $K\in{\mathbb N}$ and every $f\in L^2(X)$,
\begin{eqnarray}\label{e3.6}
{\mathcal M}^{\#}_{T_F, L, K}f(x)&=& \sup_{B\ni x} \Big({1\over V(B)}\int_{B} \big|T_F(I-e^{-r_B^mL})^K f(y)\big|^2 d\mu(y)\Big)^{1/2} ,
\end{eqnarray}
where the supremum is taken over all balls $B$ in $X$ containing $x$, and $r_B$ is the radius of $B.$ Note that
by duality it suffices  to
 prove Theorem~\ref{th3.1} for  $2< p<p'_0$. We shall  show  that if $K$ is large enough, then
\begin{eqnarray}\label{e3.7}
\big\|
{\mathcal M}^{\#}_{T_F, L, K}f \big\|_p\leq CA N^{{1\over 2}-{1\over p}}\|f\|_p,
\end{eqnarray}
where $A$ is given in \eqref{e3.2}, and
\begin{eqnarray}\label{e3.8}
N=     {\rm log}
 \bigg(2+ {\sup_{t>0}\|\phi  \delta_t F\|_{  W^{M+n+1}_2}  \over A}\bigg).
\end{eqnarray}
Once we show estimates \eqref{e3.7} and \eqref{e3.8}, it follows
 from  (b) of Lemma~\ref{le2.4} and Proposition~\ref{prop3.2} (with some $p<q<p_0'$) that
\begin{eqnarray*}
\|F(\sqrt[m]{L})f\|_p &\leq&
C \| T_Ff \|_p\\
&\leq& C \| \big({\mathcal M}\big( | T_Ff|^2 \big)\big)^{1/2} \|_p\\
&\leq& C_p \big(  \|{\mathcal M}^{\#}_{T_F, L, K}f  \|_p +\|f\|_p\big)\\
&\leq& CA N^{ {1\over 2}-{1\over p} }\|f\|_p,
\end{eqnarray*}
and this concludes the proof of Theorem~\ref{th3.1}.

Therefore it suffices to prove \eqref{e3.7}.  By Minkowski's
inequality
$$
{\mathcal M}^{\#}_{T_F, L, K}f(x)\leq {\mathscr E}_1(f)(x)+{\mathscr E}_2(f)(x),
$$
where
$$
 {\mathscr E}_1(f)(x)=
\sup_{B\ni x} \Big({1\over V(B)}\int_B  \sum_{|k+{\rm log}_2 r_B|\leq N}
\big|\big(I-e^{-r_B^mL}\big)^K(\phi^2_kF)(\sqrt[m]{L})f(y) \big|^2d\mu(y)\Big)^{1/2}
$$
and
$$ {\mathscr E}_2(f)(x)=
\sup_{B\ni x} \Big({1\over V(B)}\int_B  \sum_{|k+{\rm log}_2 r_B|> N}
\big|\big(I-e^{-r_B^mL}\big)^K(\phi^2_kF)(\sqrt[m]{L})f(y) \big|^2d\mu(y)\Big)^{1/2}.
$$
To prove estimate \eqref{e3.7} it is enough to show that
\begin{eqnarray}\label{e3.9}
\|{\mathscr E}_1(f)\|_p\leq C N^{ {1\over 2}-{1\over p} } \big\|
 \big(\sum_{  k\in{\mathbb Z}}\big|(\phi^2_kF)(\sqrt[m]{L})f
 \big|^p\big)^{1/p}\big\|_p
\end{eqnarray}
and
\begin{eqnarray}\label{e3.10}
\|{\mathscr E}_2(f)\|_p\leq CA \big\|
 \big(\sum_{ k\in{\mathbb Z}}\big|{\phi}_k (\sqrt[m]{L})f
 \big|^2\big)^{1/2}\big\|_p.
\end{eqnarray}
  Indeed, by \eqref{e3.9} and  \eqref{e3.2},
\begin{eqnarray*}
  \big\|
 \big(\sum_{  k\in{\mathbb Z}}\big|(\phi^2_kF)(\sqrt[m]{L})f
 \big|^p\big)^{1/p}\big\|_p^p&\leq& \sum_k \big\|
  (\phi^2_kF)(\sqrt[m]{L}) f\big\|_p^p  \nonumber\\
 &\leq& \sup_{k\in{\mathbb Z}}\big\|
  (\phi_kF)(\sqrt[m]{L}) \big\|_{p\to p}^p  \sum_k \big\|{\phi}_k (\sqrt[m]{L})f
 \big\|_p^p\nonumber\\
  &\leq& A^p   \big\|\big(\sum_{k\in{\mathbb Z}} |{\phi}_k (\sqrt[m]{L})f |^2\big)^{1/2}
 \big\|_p^p\nonumber\\
 &\leq& CA^p    \|f \|_p^p.
\end{eqnarray*}
Noq by \eqref{e3.10} and Lemma~\ref{le2.4},
$$
\|{\mathscr E}_2(f)\|_p\leq CA  \|
f\|_p.
$$
These estimates imply \eqref{e3.7}.

It remains to prove claims \eqref{e3.9} and \eqref{e3.10}.

\smallskip

\noindent
{{\em Proof of \eqref{e3.9}}}
Similarly as in  the proof of  \eqref{e3.5},    ${\rm (DG_m)}$ and (ii) of Lemma~\ref{le2.1}
   yields
\begin{eqnarray*}
 &&\hspace{-1cm}
 {\mathscr E}_1(f)(x) \\
&\leq&  C_K \sum_{s=1}^K\sum_{j=0}^{\infty}(1+j)^{n\over 2} e^{-c_s j^{m\over m-1}}
\Big\{ \sup_{B\ni x} {1\over V((j+1)B)} \int_{(j+1)B}
\sum_{|k+{\rm log}_2 r_B|\leq N} \big| (\phi^2_kF)(\sqrt[m]{L})f \big|^2 d\mu  \Big\}^{1/2}\nonumber \\
&\leq&  C_K N^{ {1\over 2}-{1\over p} }\sum_{s=1}^K \sum_{j=0}^{\infty}(1+j)^{n\over2}e^{-c_s j^{m\over m-1}}
\sup_{B\ni x}\Big\{  {1\over V((j+1)B)} \int_{(j+1)B}\Big(
\sum_{k\in{\mathbb Z}} \big| (\phi^2_kF)(\sqrt[m]{L})f \big|^p\Big)^{2/p} d\mu  \Big\}^{1/2}\nonumber \\
&\leq&  C_KN^{ {1\over 2}-{1\over p} } \Big({\mathcal M}\Big(  \Big\{ \sum_{k\in{\mathbb Z}}\big|
(\phi^2_kF)(\sqrt[m]{L})f \big|^p\Big\}^{2/p}\Big) \Big)^{1/2}(x).
\end{eqnarray*}
Then  by $L^{p/2}$-boundedness of ${\mathcal M}$
\begin{eqnarray*}
\|{\mathscr E}_1(f)\|_p
&\leq& C N^{ {1\over 2}-{1\over p} } \big\|
 \big(\sum_{  k\in{\mathbb Z}}\big|(\phi^2_kF)(\sqrt[m]{L})f
 \big|^p\big)^{1/p}\big\|_p.
\end{eqnarray*}
This shows  \eqref{e3.9}.

\medskip

\noindent
{{\em Proof of \eqref{e3.10}}}.
We need a further decomposition  of ${\mathscr E}_2(f).$ Note that
$${\mathscr E}_2(f)(x)\leq {\mathscr E}_{21}(f)(x)+{\mathscr E}_{22}(f)(x),
$$ where
$${\mathscr E}_{21}(f)(x)=\sup_{B\ni x}\Big( {1\over V(B)}\int_B  \sum_{k\in{\mathbb Z}}
\big| (1-e^{-r_B^mL})^K({ \phi}_kF)(\sqrt[m]{L})\big(P_{2B}{\phi}_k (\sqrt[m]{L})f\big)(y) \big|^2d\mu(y)\Big)^{1/2}
$$
and
$${\mathscr E}_{22}(f)(x)=\sup_{B\ni x}\Big( {1\over V(B)}\int_B  \sum_{|k+{\rm log}_2 r_B|> N}
\big|\big(I-e^{-r_B^mL}\big)^K({ \phi}_kF)(\sqrt[m]{L}) \big(P_{X\backslash 2B}{\phi}_k (\sqrt[m]{L})f\big)(y) \big|^2 d\mu(y)\Big)^{1/2}.
$$
We first estimate the term ${\mathscr E}_{21}(f)$. By  \eqref{e3.2}
\begin{eqnarray*}
&&\hspace{-1 cm}{\mathscr E}_{21}(f)(x)\\
&\leq&
\sup_{B\ni x}\Big({1\over V(B)}\sum_{k\in{\mathbb Z}} \big\|(1-e^{-r_B^mL})^K
({  \phi}_kF)(\sqrt[m]{L})\big(P_{2B}{\phi}_k (\sqrt[m]{L})f\big)\big\|_2^2
 \Big)^{1/2}\nonumber\\
  &\leq& C\sup_{r_B>0} \sup_{k\in{\mathbb Z}}\big\|
  (1-e^{-r_B^mL})^K({ \phi}_kF)(\sqrt[m]{L}) \big\|_{2\to 2}
  \sup_{B\ni x}\Big({1\over V(2B)}\int_{2B} \sum_{k\in{\mathbb Z}}
 | {\phi}_k (\sqrt[m]{L})f(y) \big|^2 d\mu(y)\Big)^{1/2}\nonumber\\
 &\leq& C A\big( {\mathcal M}\big( \sum_{k\in{\mathbb Z}} |\phi_k(\sqrt[m]{L})f|^2\big)\big)^{1/2}(x).
\end{eqnarray*}
By $L^{p/2}$-boundedness of ${\mathcal M}$
\begin{eqnarray*}
 \|{\mathscr E}_{21}(f)\|_p &\leq& CA\big\|\big( {\mathcal M}( \sum_{k\in{\mathbb Z}} |\phi_k(\sqrt[m]{L})f|^2\big)\big)^{1/2}\big\|_p\\
  &\leq& CA \big\|
 \big(\sum_{  k\in{\mathbb Z}}\big|{\phi}_k (\sqrt[m]{L})f
 \big|^2\big)^{1/2}\big\|_p.
\end{eqnarray*}
Next, we consider  the term ${\mathscr E}_{22}(f)$.
Observe that
 \begin{eqnarray*}
 {\mathscr E}_{22}(f)(x) \le
\sup_{B\ni x}     \sum_{j=2}^{\infty} \sum_{|k+  \log_2 r_B|>N}  V(B)^{-1/2}
\big\|P_{B}  (1- e^{-r^m_BL})^K  (\phi_kF)(\sqrt[m]{L})\big( P_{A(x_B, r_B, j)}
 {\phi}_k(\sqrt[m]{L})f\big)\big\|_2.
\end{eqnarray*}
By conditions \eqref{e1.2} and \eqref{e1.3},
 \begin{eqnarray*}
 &&\hspace{-2cm}\big\|P_{B}   (1- e^{-r^m_BL})^K  (\phi_kF)(\sqrt[m]{L}) \big(P_{A(x_B, r_B, j)}
 {\phi}_k(\sqrt[m]{L})f\big)\big\|_2\\[1pt]
  &\leq&  Cj^{n/2} V(B)^{1/2}
\big\|P_{B}  (1- e^{-r^m_BL})^K  (\phi_kF)(\sqrt[m]{L}) P_{A(x_B, r_B, j)} \big\|_{2\to 2} \times \\
 &&\hspace{3cm} \times \Big({1\over V((j+1)B)} \int_{(j+1)B}|{\phi}_k(\sqrt[m]{L})f(y)|^2 d\mu(y)\Big)^{1/2}.
\end{eqnarray*}
 To continue, we note that the function
$ (1- e^{-r^m_B\lambda^m})^K \phi_k(\lambda)F(\lambda)$ is supported in $  [2^{k-1}, \ 2^{k+1}]$. Now,
if $k$ is an integer greater than $M+n+1$, then for Sobolev space
${  W^{M+n+1}_2}({\mathbb R})$
\begin{eqnarray*}
\|\delta_{2^{k+1}}\big((1- e^{-r^m_B\lambda^m})^K \phi_k(\lambda)F(\lambda)\big) \|_{  W^{M+n+1}_2} \\
 \leq  \| (1-e^{-(2^{k+1}r_B\lambda)^m})^K\phi(\lambda)\delta_{2^{k+1}}F(\lambda) \|_{  W^{M+n+1}_2}\\
 \leq  \| (1-e^{-(2^{k+1}r_B\lambda)^m})^K\|_{C^k([{1\over 2}, 1])} \|\phi\delta_{2^{k+1}}F \|_{  W^{M+n+1}_2}\\
 \leq  C \min\big\{ 1, \ \big(2^k r_B\big)^{mK} \big\}  \sup_{t>0}\|\phi  \delta_t F\|_{  W^{M+n+1}_2}.
\end{eqnarray*}
By  Lemma~\ref{le2.2}  for every $M>0$
\begin{eqnarray*}
 \big\|P_{B}   (1- e^{-r^m_BL})^K  (\phi_kF)(\sqrt[m]{L}) P_{A(x_B, r_B, j)}\big\|_{2\to 2}\\
 \leq  C j^{-M}\big(2^k r_B\big)^{-M-n} \|\delta_{2^{k+1}}\big((1- e^{-r^m_B\lambda^m})^K
 \phi_k(\lambda)F(\lambda)\big) \|_{  W^{M+n+1}_2}\\
  \leq  Cj^{-M}\min\big\{ \big(2^k r_B\big)^{-M-n}, \ \big(2^k r_B\big)^{mK-M-n} \big\}  \sup_{t>0}\|\phi  \delta_t F\|_{  W^{M+n+1}_2}.
\end{eqnarray*}
 Therefore, we sum  a geometrical series to  obtain that if $M>n/2+1$ and  $mK-M-n>1$ in \eqref{e3.6}, then
 \begin{eqnarray*}
 {\mathscr E}_{22}(f)(x)
   \leq   C  \sup_{t>0}\|\phi  \delta_t F\|_{  W^{M+n+1}_2}
    \sup_{B\ni x} \sum_{j=2}^{\infty} j^{n/2-M} \\
	\times
	 \sum_{|k+ \log_2 r_B |>N}  \min\big\{ \big(2^k r_B\big)^{-M-n}, \ \big(2^k r_B\big)^{ mK-M-n}\big\}
	\Big({1\over V((j+1)B)} \int_{(j+1)B}|{\phi}_k(\sqrt[m]{L})f |^2 d\mu
	   \Big)^{1/2}\\
    \leq   C 2^{-N} \sup_{t>0}\|\phi  \delta_t F\|_{  W^{M+n+1}_2}
   \sum_{j=2}^{\infty} j^{n/2-M}
	\sup_{B\ni x}\Big({1\over V((j+1)B)} \int_{(j+1)B}\sum_{k\in{\mathbb Z}}|{\phi}_k(\sqrt[m]{L})f |^2 d\mu \Big)^{1/2} \nonumber\\
   \leq   CA \Big({\mathcal M}\Big(\sum_{k\in{\mathbb Z}}|\phi_k(\sqrt[m]{L})f|^2\Big)\Big)^{1/2} (x),
\end{eqnarray*}
where   we  use the fact that by condition \eqref{e3.8},
  $2^{-N} \sup_{t>0}\|\phi  \delta_t F\|_{  W^{M+n+1}_2}\leq CA$.

Again, by $L^{p/2}$-boundedness of ${\mathcal M}$
 \begin{eqnarray*}
 \|{\mathscr E}_{22}(f)\|_p
     &\leq &     CA \big\| \big(\sum_{k\in{\mathbb Z}}|\phi_k(\sqrt[m]{L})f|^2\big)^{1/2} \big\|_p.
\end{eqnarray*}
This shows \eqref{e3.10} and  completes the proof of Theorem~\ref{th3.1}.  \hfill{} $\Box$

\bigskip

\medskip

By a classical dyadic decomposition of $F$, we can write
$F(\sqrt[m]{L})$ as the sum $\sum F_j(\sqrt[m]{L})$. Then we  apply Theorem  \ref{th3.1}   to estimate  $\| F_j(\sqrt[m]{L})\|_{r \to r}$.
However, as mentioned in the introduction, this does not automatically imply that the
operator $F(\sqrt[m]{L})$ acts boundedly on $L^r$. See \cite{Carb1, S3,SS} where this problem is discussed in the Euclidean case.

We shall now discuss a a criterion which guarantee boundedness of $F({L})$
under assumption that  multipliers supported in dyadic
intervals are uniformly bounded.
In Section \ref{sec4} we describe results concerning multipliers supported in dyadic
intervals.

\medskip

 \begin{theorem}\label {th3.3}
Let  $L$  be  a non-negative self-adjoint operator $L$ on $L^2(X)$ satisfying
Davies-Gaffney estimates ${\rm (DG_m)}$  and condition ${\rm (G_{p_0,2,m})}$ for some
  $1\leq p_0< 2$.
Assume that for any bounded Borel function $H$  such that supp $H\subseteq [{1/4}, 4]$, the following
condition holds:
\begin{eqnarray}\label{e3.11}
\sup_{t>0}\|H(t\sqrt[m]{L}) \|_{p\to p}   \leq  C\|  H\|_{W^{\alpha}_q}
\end{eqnarray}
for some $p\in (p_0, 2)$, $\alpha>n(1/p-1/2)$, and $1\leq q\leq \infty$.
  Then for any bounded Borel function $F$  such that
\begin{eqnarray}\label{e3.12}
\sup_{t>0}\|\phi\delta_tF\|_{W^{\alpha}_q} < \infty
\end{eqnarray}
for some $\alpha>\max\{n(1/p-1/2),1/q\}$,  the operator
$F( {L})$ is bounded on $L^r(X)$ for all $p<r<p'$. In addition,
\begin{eqnarray}\label{e3.13}
 \|F( {L}) \|_{r\to r} \leq  C\sup_{t>0}\|\phi\delta_tF\|_{W^{\alpha}_q}.
\end{eqnarray}
\end{theorem}

\begin{proof}
Observe that  $ \| F\|_{W^q_\alpha} \sim  \| G\|_{W^q_\alpha}$
where $G(\lambda)=F(\sqrt[m]{\lambda})$.
 For
this reason, we can replace $F(L)$ by  $F(\sqrt[m]{L})$ in the proof.
 Let   $\psi$ be a $C^{\infty}$-function,
supported in $\{|\xi|\leq 1/8\}$, $\int \psi(\xi)d\xi=1.$
Further set $\psi_{\ell}=2^{\ell  }\psi(2^{\ell}\cdot)$, $\theta_{\ell}=  \psi_{\ell} -\psi_{\ell-1}\
(\ell\geq 1), \theta_0=\psi_0. $
We write

\begin{eqnarray}\label{e3.14}
F &=&\sum_{j\in {\mathbb Z}} \phi(2^{-j}\cdot) F\nonumber\\
&=& \sum_{j\in {\mathbb Z}}\sum_{\ell\geq 0} \big[\theta_\ell\ast \big(\phi F(2^j \cdot)\big)\big] (2^{-j}\cdot)\nonumber\\
&=&   \sum_{\ell\geq 0} F_\ell
\end{eqnarray}
and so for every $p<r< p'$
$$
\|F(\sqrt[m]{L})\|_{r\to r}\leq \sum_{\ell\geq 0}\|F_{\ell}(\sqrt[m]{L})\|_{r\to r}.
$$

 To estimate   terms  $\|F_{\ell}(\sqrt[m]{L})\|_{r\to r}, \ell\geq 0$ we shall apply Theorem~\ref{th3.1} .  Firstly,
we claim that for $p<r< p'$,
there exists some $\eta_r>0$ such that
\begin{eqnarray}\label{e3.15}
\sup_{t>0}\|\big(\phi \delta_tF_{\ell}\big)(\sqrt[m]{L})\|_{r\to r}\leq C_r2^{- \eta_r\ell } \sup_{t>0}\|\phi\delta_t F\|_{W^{\alpha}_q}.
\end{eqnarray}
By duality we may assume that $p<r\leq 2$.  Observe  that
  $ \theta_\ell\ast \big(\phi F(2^j\cdot)\big)$ is supported in $\{\xi: \ {1\over 4}\leq |\xi|\leq 4 \}.$
  If $\ell\geq 1$, we have that for $t\in [2^k, 2^{k+1}]$,
  \begin{eqnarray}\label{e3.16}
 \big(\phi \delta_tF_{\ell}\big)(\sqrt[m]{L})
= \sum_{j=k-4}^{k+4}  \phi(\sqrt[m]{L})
\big[\theta_\ell\ast \big(\phi F(2^j\cdot)\big)\big] (2^{-j}t\sqrt[m]{L}).
\end{eqnarray}
 Now  we recall that if $1\leq q\leq \infty$ and $\alpha-1/q>0$, then
$$
W^{\alpha}_q \subseteq  B^{\alpha}_{q, \, \infty} \subseteq  B^{\alpha-{1\over q}}_{\infty, \, \infty}
\subseteq
\Lambda_{\min\{\alpha-{1\over q}, \, {1\over 2}\}}
$$
 and  $\|F\|_{\Lambda_{\min\{\alpha-{1\over q}, \, {1\over 2}\}}}\leq C\|F\|_{W^{\alpha}_q}$.
See, e.g., \cite[Chap.~VI ]{BL} for more details.  Hence
\begin{eqnarray*}
\|\phi \delta_{t}F \|_{\Lambda_{\min\{\alpha-1/q, \, 1/2\}}}\leq
C\|\phi \delta_{t}F\|_{W^{\alpha}_q}.
\end{eqnarray*}
This  implies that 
$$
\|\theta_\ell\ast \big(\phi F(2^j \cdot)\big)\|_{\infty}\leq C2^{-\ell \epsilon}\sup_{t>0}\|\phi \delta_{t}F\|_{W^{\alpha}_q}
$$
with $\epsilon=\min\{\alpha-1/q, \, 1/2\} $. Hence
 $$\|\big[\theta_\ell\ast \big(\phi F(2^j\cdot)\big)\big] (2^{-j}t\sqrt[m]{L})\|_{2\to 2} \leq C2^{-\ell \epsilon}
 \sup_{t>0}\|\phi \delta_{t}F\|_{W^{\alpha}_q}.
 $$
By \eqref{e3.16} and the fact  that $\|\phi(\sqrt[m]{L})\|_{2 \to 2}\leq C$
  \begin{eqnarray}\label{e3.18}
\|\big(\phi \delta_tF_{\ell}\big)(\sqrt[m]{L})\|_{2\to 2}  &\leq &  C2^{-\ell \epsilon}
 \sup_{t>0}\|\phi \delta_{t}F\|_{W^{\alpha}_q}.
\end{eqnarray}
Note that for each $\ell$, the function $ \theta_\ell\ast \big(\phi F(2^j\cdot)\big)$ is supported in $\{\xi: \ {1\over 4}\leq |\xi|\leq 4 \}.$
By \eqref{e3.11}
\begin{eqnarray*}
\| \big[\theta_\ell\ast \big(\phi F(2^j\cdot)\big)\big] (2^{-j} t\sqrt[m]{L})\|_{p\to p}
&\leq& C\|  \theta_\ell\ast \big(\phi F(2^j\cdot)\big) \|_{W^{\alpha}_q}\\
&\leq & C\sup_{t>0}\|\phi \delta_{t}F\|_{W^{\alpha}_q}.
\end{eqnarray*}
 (a) of Lemma~\ref{le2.4} shows   that  $\|\phi(\sqrt[m]{L})\|_{p \to p}\leq C.$ By \eqref{e3.16}
   \begin{eqnarray*}
\|\big(\phi \delta_tF_{\ell}\big)(\sqrt[m]{L})\|_{p\to p}  &\leq &   C \sup_{t>0}\|\phi \delta_{t}F\|_{W^{\alpha}_q}.
\end{eqnarray*}
Then it follows from the interpolation theorem   that for every $r\in (p, 2)$,
$$
\|\big(\phi \delta_tF_{\ell}\big)(\sqrt[m]{L})\|_{r\to r}\leq C_r2^{- \eta_r\ell } \sup_{t>0}\|\phi\delta_t F\|_{W^{\alpha}_q}
$$
with $\eta_r= {\epsilon\big({1/r}-{1/p}\big)/\big({1/2}-{1/p }\big)}$, and this  shows \eqref{e3.15}.

By Theorem~\ref{th3.1} for every $M>1+{n\over 2}$,
    \begin{eqnarray}\label{e3.19}\hspace{1cm}
\|  F_{\ell}(\sqrt[m]{L})\|_{r\to r}  &\leq &   C2^{-\eta_r \ell} \sup_{t>0}\|\phi \delta_{t}F\|_{W^{\alpha}_q}
 \bigg\{ {\rm log}  \Big(2+ { \sup_{t>0}\|\phi  \delta_t F_{\ell}\|_{  W^{M+n+1}_2}
  \over 2^{-\eta_r \ell}\sup_{t>0}\|\phi \delta_{t}F\|_{W^{\alpha}_q}}\Big)
 \bigg\}^{|{1\over r}-{1\over 2}|}.
\end{eqnarray}
Let $s\geq 1$ such that ${1\over q} +{1\over s}={3\over 2}$. The Plancherel theorem and
Young's inequality yields 
that  if $\ell\geq 1$, and   $t\in [2^k, 2^{k+1}]$,
  \begin{eqnarray*}
\| \phi(\cdot)\big[\theta_\ell\ast \big(\phi F(2^j\cdot)\big)\big] (2^{-j}t\cdot)\|_{  W^{M+n+1}_2}
 \leq
\|  (1+\xi^2)^{ {M+n+1\over 2}}\widehat{\phi F(2^j\cdot)}(\xi)\widehat \theta_{\ell} (\xi) \|_{2}\\
= \|{\mathscr F}^{-1}\big((1+\xi^2)^{\alpha\over 2}\widehat{\phi F(2^j\cdot)}(\xi)\big) \ast
{\mathscr F}^{-1}\big((1+\xi^2)^{{M+n+1-\alpha\over 2}}\widehat \theta_{\ell}(\xi)\big)\|_{2}\\
\leq  \|\phi F(2^j\cdot)\|_{W^{\alpha}_q} \|\theta_{\ell}\|_{W^{M+ n+1 -\alpha}_s} \\
 \leq C2^{\ell ( n+ M-\alpha+{1\over 2}+{1\over q}) } \sup_{t>0}\|\phi \delta_{t}F\|_{W^{\alpha}_q}.
\end{eqnarray*}
Hence
\begin{eqnarray}\label{e3.20}
\sup_{t>0}\|\phi  \delta_t F_{\ell}\|_{  W^{M+n+1}_2}
\leq C2^{\ell ( n+ M-\alpha+{1\over 2} +{1\over q}) } \sup_{t>0}\|\phi \delta_{t}F\|_{W^{\alpha}_q}.
\end{eqnarray}
Substituting \eqref{e3.20} into \eqref{e3.19}, we get
     \begin{eqnarray*}
\|  F_{\ell}(\sqrt[m]{L})\|_{r\to r}  &\leq &   C2^{-\eta_r \ell} (1+\ell)^{|{1\over r}-{1\over 2}|}
\sup_{t>0}\|\phi \delta_{t}F\|_{W^{\alpha}_q}.
\end{eqnarray*}

Analogously, $\|  F_{0}(\sqrt[m]{L})\|_{r\to r}   \leq     C
\sup_{t>0}\|\phi \delta_{t}F\|_{W^{\alpha}_q}.$
Summing a geometrical series we obtain
$
\|  F(\sqrt[m]{L})\|_{r\to r}   \leq     C \sup_{t>0}\|\phi \delta_{t}F\|_{W^{\alpha}_q}.
$
This completes the proof.
\end{proof}

\bigskip

\section{Dyadiclly  supported (non-singular)  spectral multipliers.}\label{sec4}
\setcounter{equation}{0}

In this section,
we will  show  that restriction type conditions can be used to study spectral  multipliers
corresponding to  functions supported in dyadic intervals. We assume that 
   $(X, d, \mu)$ is  a  metric measure  space satisfying the doubling property
 and   $n$ is the  doubling dimension from  condition \eqref{e1.3}.

\subsection{Operators with continuous spectrum}
 Consider a non-negative self-adjoint operator $L$  and numbers
$p $ and $q$ such that $1\leq p< 2$ and $1\leq
q\leq\infty$.  Given $R_0\geq 0$, we say that
operator $L$ satisfies the
    {\it  local Stein-Tomas restriction type condition} if:  for any $R>R_0$ and
 for all Borel functions $F$ such that $\supp F \subset [R/2, R]$,
$$
\big\|F(\SL)P_{B(x, r)} \big\|_{p\to 2} \leq CV(x,
r)^{{1\over 2}-{1\over p}} \big( Rr \big)^{n({1\over p}-{1\over
2})}\big\|\delta_RF\big\|_{q}
\leqno{\rm (ST^{q}_{p, 2, m})_{R_0}}
$$
  for all
  $x\in X$ and all $r\geq 1/R$.

 The condition ${\rm (ST^{q}_{p, 2, m})_{R_0}}$ is a small modification of the restriction type condition ${\rm (ST^{q}_{p, 2, m})}$
   Namely here we consider function supported in the interval $[R/2, R]$ rather then $[0,  R]$, which
  allows us to study localized version of spectral multipliers, see   Theorems~\ref{th4.5} and  ~\ref{th5.8} below.

    Note that condition   ${\rm (ST^{q}_{p, 2, m})}$ implies  ${\rm (ST^{q}_{p, 2, m})_{R_0}}$  for all
  $R_0 \ge 0$.
  If in addition we assume that $\chi_{\{ 0\}}(\sqrt L)=0$ then for $R_0=0$ condition ${\rm (ST^{q}_{p, 2, m})_{R_0}}$
  implies ${\rm (ST^{q}_{p, 2, m})}$.


We say that $L$ satisfies {\it $L^p$ to $L^{p'} $ restriction estimates} if  there exists  $\lambda_0\geq 0$ such that
the spectral
measure $dE_{\sqrt[m]{L}}(\lambda)$ maps  $L^p(X)$ to $L^{p'}(X)$ for some $p<2$, with an operator norm estimate
$$
\big\|dE_{\sqrt[m]{L}}(\lambda)\big\|_{p\to p'}\leq C \lambda^{n({1\over p}-{1\over p'})-{1}}\ \ \ {\mbox{for all $\lambda\ge \lambda_0$},}
\leqno{\rm (R_{p,m})_{\lambda_0}}
$$
 where $n$ is
   as in  condition (\ref{e1.3}) and
 $p'$ is conjugate of $p$, i.e., ${1/p} +{1/p'}=1$.

 \medskip

\begin{proposition}\label{prop4.4} Let $1\le p< 2$ and $R_0\geq 0$.
Suppose that  there exists a  constant $C>0 $  such that
$C^{-1}r^n \leq V(x, r)\leq C r^n$
for all $x\in X $ and $r>0$.  Then
conditions ${\rm (R_{p,m})_{R_0/2}}$, ${\rm (ST^{2}_{p, 2,m})_{R_0}}$ and
 ${\rm (ST^{1}_{p, p',m})_{R_0}}$
are equivalent.
\end{proposition}

 \begin{proof} The proof is similar to  that of  Proposition 2.4 of \cite{COSY} with minor modifications.  We give a brief argument
of this proof for completeness and  readers' convenience.

We first  show the implication
${\rm (R_{p,m})_{R_0/2}}\Rightarrow {\rm (ST^{1}_{p, p',m})_{R_0}}$.
Suppose that $F$  is a  Borel function    such that $\supp F \subset [R/2,  R]$
for some $R> R_0$.  Then by ${\rm (R_{p,m})_{R_0/2}}$
\begin{eqnarray*}
\big\|F(\SL) P_{B(x,r)} \big\|_{p\to p'}
 &\leq&  \int_0^\infty
|F(\l)| \|dE_{\sqrt[m]{L}}(\l)\|_{p\to
p'}d\l\\
 &\leq&  C\int_{R/2}^{R} |F(\l)| \lambda^{n({{\frac{1}{p}}-{\frac{1}{p'}}})-1}d\l \\
 &\leq&  CR^{n({{\frac{1}{p}}-{\frac{1}{p'}}})-1}\int_{R/2}^{R} |F(\l)|  d\l \\
 &\leq&  C V(x,r)^{{\frac{1}{p'}}-{\frac{1}{p}}}  (rR)^{n({{\frac{1}{p}}-{\frac{1}{p'}}})} \|\delta_{R}F\|_{1},
\end{eqnarray*}
where in the last inequality we used the assumption that  $V(x, r) \le C  r^n$.

Next we prove that ${\rm (ST^{1}_{p,p',m})_{R_0}}\Rightarrow {\rm (ST^{2}_{p, 2,m})_{R_0}}$. Note that
$  V(x, r)\sim r^n$ for every $x\in X$ and $r>0.$
Letting $r \to \infty$ we obtain from ${\rm (ST^{1}_{p,p',m})_{R_0}}$
$$ \big\|F(\SL)\big\|_{p\to p'} \le CR^{n({{1\over p}-{1\over p'}})}\|\delta_{R}F\|_{1}, \ \ R>R_0.$$
By $T^{\ast} T$ argument
$$
 \big\|F(\SL)\big\|^2_{p\to 2}
 =   \big\| |F|^2(\SL)\big\|_{p\to p'}
 \leq  CR^{2n({{1\over p}-{1\over 2}})}\|\delta_{R}F\|^2_{2}.
$$
Hence
\begin{eqnarray*}
 \big\|F(\SL)P_{B(x, r)}\big\|_{p\to 2} \leq  \big\|F(\SL)\big\|_{p \to 2}
 &\leq& CV(x,r)^{{\frac{1}{2}}-{\frac{1}{p}}}(Rr)^{n({{\frac{1}{p}}-{\frac{1}{2}}})}\|\delta_{R}F\|_{2}.
\end{eqnarray*}
 This gives ${\rm (ST^{2}_{p, 2, m})_{R_0}}$.

 Now, we   prove the remaining implication
${\rm (ST^{2}_{p, 2, m})_{R_0}}\Rightarrow  {\rm (R_{p,m})_{R_0/2}}$. By
volume estimate   $V(x, r)\geq C^{-1} r^n$  and condition  ${\rm (ST^{2}_{p, 2, m})_{R_0}}$
\begin{eqnarray}\label{e4.3}
\big\|F(\SL)P_{B(x, r)} \big\|_{p\to 2} \leq
C R^{n({1\over p}-{1\over 2})}\big\|\delta_{R}F\big\|_{2}
\end{eqnarray}
  for all Borel functions $F$ such that $\supp F \subset [R/2,  R]$,  all $R>R_0$, all
  $x\in X$ and $r\geq 1/R$. Taking the limit $r\to \infty$ gives  \smallskip
\begin{eqnarray}\label{e4.4}
\big\|F(\SL)  \big\|_{p\to 2} \leq
C R^{n({1\over p}-{1\over 2})}\big\|\delta_{R}F\big\|_{2}.
 \end{eqnarray}
 \noindent
Let $\epsilon<R/4$. Putting  $F=\chi_{(\lambda-\varepsilon, \lambda+\eps]}$ and $R=\lambda+\epsilon$  in  (\ref{e4.4}) yields
 \begin{eqnarray*}
 \Big\|\varepsilon^{-1}{E_{\sqrt[m]{L}}(\lambda-\eps, \lambda+\varepsilon]  } \Big\|_{p\to p'}
 &=&\varepsilon^{-1}\Big\| {E_{\sqrt[m]{L}}(\lambda-\eps, \lambda+\varepsilon]  } \Big\|^2_{p\to 2} \\
 &\leq&
C \varepsilon^{-1}(\lambda+ \eps)^{2n({1\over p}-{1\over 2})}
\big\| \chi_{({ \lambda-\eps\over  \lambda+\epsilon}, \, 1]}\big\|^2_{2} \\
 &\leq& C (\lambda+\epsilon)^{n({1\over p}-{1\over p'})-1}.
 \end{eqnarray*}
 Taking $\varepsilon\to 0$ yields condition ${\rm (R_{p, m})_{R_0/2}}$ (see Proposition 1, Chapter XI, \cite{Yo}).
\end{proof}

\bigskip

 The  following result  and its  proof  are
inspired  by  Theorem 1.1 of  \cite{GHS}.  See also Theorem 3.1 of \cite{COSY}.

 \begin{theorem}\label{th4.5} Suppose that  $(X, d, \mu)$ and
  a non-negative self-adjoint operator $L$ acting on $L^2(X)$ satisfies  estimates ${\rm (DG_m)}$
  and ${\rm (G_{p, 2, m})}$ for some $1\leq p<2$. Next assume that  condition  ${\rm (ST^q_{p, 2, m})_{R_0}}$ holds   for some $R_0 \ge 0$ and
  and $1\leq  q\leq \infty$ and that  $F$ is a bounded Borel  function such that  $\supp F\subseteq [1/4,4]$
 and $F \in W^{\alpha}_q({\mathbb R})$
for some $\alpha>n(1/p-1/2)$.

Then for every $ p < r\leq 2$, $F(t \sqrt[m]{L})$ is bounded on $L^r(X)$,
\begin{equation}
\label{e4.6}
\sup_{t < 1/(8R_0)}\|F(t \sqrt[m]{L})\|_{r\to r} \leq
C_r\|F\|_{W^{\alpha}_q}
\end{equation}
and
\begin{equation}
\label{e4.7}
\sup_{t\geq  1/(8R_0)}\|F(t \sqrt[m]{L})\|_{r\to r} \leq
C_r\|F\|_{W^{\alpha}_{\infty}}.
\end{equation}
\end{theorem}

\medskip

 \noindent
 {\em Proof}.
Let   $\phi \in C_c^{\infty}(\mathbb R) $   be  a function such that $\supp \phi\subseteq \{ \xi: 1/4\leq |\xi|\leq 1\}$ and
 $
\sum_{\ell\in \ZZ} \phi(2^{-\ell} \lambda)=1$ for all
${\lambda>0}.
$
Set $\phi_0(\lambda)= 1-\sum_{\ell=0}^{\infty} \phi(2^{-\ell} \lambda)$,
\begin{eqnarray}\label{e4.8}
G^{(0)}(\lambda)=\frac{1}{2\pi}\int_{-\infty}^{+\infty}
 \phi_0(\tau) \hat{G}(\tau) e^{i\tau\lambda} \;d\tau
\end{eqnarray}
and
\begin{eqnarray}\label{e4.9}
G^{(\ell)}(\lambda) =\frac{1}{2\pi}\int_{-\infty}^{+\infty}
 \phi(2^{-\ell}\tau) \hat{G}(\tau) e^{i\tau\lambda} \;d\tau,
\end{eqnarray}
where $G(\lambda)=F(\sqrt[m]\lambda) e^{\lambda}$.  Note that in by the Fourier inversion formula
$$
G(\lambda)=\sum_{\ell=0}^{\infty}G^{(\ell)}(\lambda).
$$
Then
\begin{eqnarray}\label{e4.10}
F(\sqrt[m]\lambda)= G(\lambda)e^{-\lambda}
&=&\sum_{\ell=0}^{\infty}G^{(\ell)}(\lambda)e^{-\lambda}
=:\sum_{\ell=0}^{\infty}F^{(\ell)}(\sqrt[m]\lambda)
\end{eqnarray}
so
\begin{eqnarray}\label{e4.11}
\big\|F(t\SL)\big\|_{r\to r } &\le& \sum_{\ell=0}^{\infty}\big\|F^{(\ell)}(t\sqrt[m]{L}) \big\|_{r\to r}, \ \ \ \ r\in (p, 2).
\end{eqnarray}

Next we   fix  $ \varepsilon>0$  such that
 $
 2n\varepsilon ({1/p}-{1/2}) \leq  \alpha-n({1/p}-{1/2}).
 $
For every $t>0$ and every  $\ell$ set $\rho_{\ell}=2^{\ell(1+\varepsilon)}t$.
Then  we choose a sequence $(x_n)  \in X$ such that
$d(x_i,x_j)> \rho_{\ell}/10$ for $i\neq j$ and $\sup_{x\in X}\inf_i d(x,x_i)
\le \rho_{\ell}/10$. Such sequence exists because $X$ is separable.
Now set $B_i=B(x_i, \rho_{\ell})$ and define $\widetilde{B_i}$ by the formula
$$\widetilde{B_i}=\bar{B}\left(x_i,\frac{\rho_{\ell}}{10}\right)\setminus
\bigcup_{j<i}\bar{B}\left(x_j,\frac{\rho_{\ell}}{10}\right),$$
where $\bar{B}\left(x, \rho_{\ell}\right)=\{y\in X \colon d(x,y)
\le \rho_{\ell}\}$. Note that for $i\neq j$,
 $B(x_i, \frac{\rho_{\ell}}{20}) \cap B(x_j, \frac{\rho_{\ell}}{20})=\emptyset$.

Observe that for every $k\in{\mathbb N}$,
\begin{eqnarray}\label{e4.12}
\sup_i\#\{j:\;d(x_i,x_j)\le  2^k\rho_{\ell}\} &\le&
  \sup_{d(x,y) \le 2^k\rho_{\ell}}  {V(x, 2^{k+1}\rho_{\ell})\over
    V(y, \frac{\rho_{\ell}}{20})}\nonumber\\
   &\le& C 2^{kn}
  \sup_y  {V(y, 2^{k+2}\rho_{\ell})\over
  V(y, \frac{\rho_{\ell}}{20})}
    \le C 2^{kn}.
\end{eqnarray}
Set
$
\D_{\rho_{\ell}}=\{ (x,\, y)\in X\times X: {d}(x,\, y) \le \rho_{\ell} \}.
$
It is not difficult to see that
\begin{eqnarray}\label{e4.13}
\D_{\rho_{\ell}}  \subseteq \bigcup_{\{i,j:\, d(x_i,x_j)<
 2 \rho_{\ell}\}} \widetilde{B}_i\times \widetilde{B}_j \subseteq \D_{4 \rho_{\ell}}.
\end{eqnarray}

Now let  $\psi \in C_c^\infty(1/16, 16)$ be a such function
that $\psi(\lambda)=1$ for $\lambda \in (1/8,8)$,
and  we decompose
\begin{eqnarray}\label{e4.14}
 F^{(\ell)}(t\sqrt[m]{L})f
&=& \sum_{i,j:\, {d}(x_i,x_j)< 2\rho_{\ell}} P_{\widetilde B_i}\big[\psi F^{(\ell)} (t\sqrt[m]{L}) \big]P_{\widetilde
B_j}f\nonumber\\
  &&+ \sum_{i,j:\, {d}(x_i,x_j)< 2\rho_{\ell}} P_{\widetilde B_i}\big[(1-\psi)
F^{(\ell)} (t\sqrt[m]{L}) \big] P_{\widetilde
B_j}f\nonumber\\
&&+  \sum_{i,j:\, {d}(x_i,x_j)\geq 2\rho_{\ell}} P_{\widetilde B_i}F^{(\ell)}(t\sqrt[m]{L}) P_{\widetilde
B_j}f=I+ I\!I +I\!I\!I.
\end{eqnarray}

\medskip

\noindent
{{\em Estimate for  {\it I}}} .   \ By
 H\"older's inequality,
\begin{eqnarray}\label{e4.15}
 \| \sum_{i,j:\, {d}(x_i,x_j)< 2\rho_{\ell}} P_{\widetilde B_i}\big(\psi F^{(\ell)}\big) (t\sqrt[m]{L}) P_{\widetilde
B_j}f\|_{r}^r
 =\sum_i \|\sum_{j:\,{d}(x_i,x_j)<
2\rho_{\ell}} P_{\widetilde B_i}\big(\psi F^{(\ell)}\big)(t\sqrt[m]{L})P_{\widetilde B_j}f\|_{r}^r  \nonumber \\
\le C   \sum_i \sum_{j:\,{d}(x_i,x_j)< 2\rho_{\ell}} \| P_{\widetilde
B_i}\big(\psi F^{(\ell)}\big)(t\sqrt[m]{L})P_{\widetilde B_j}f\|_{r}^r
\nonumber \\
\le C  \sum_i   \sum_{j:\,{d}(x_i,x_j)< 2\rho_{\ell}}
V(\widetilde{B}_i)^{r({1\over r}-{1\over 2})} \|P_{\widetilde
B_i}\big(\psi F^{(\ell)}\big)(t\sqrt[m]{L})P_{\widetilde B_j}f\|_{2}^r
\nonumber \\ \le C   \sum_j     V( {B}_j)^{r({1\over r}-{1\over
2})} \|\big(\psi F^{(\ell)}\big)(t\sqrt[m]{L})P_{ \widetilde B_j}f\|_{2}^r
\nonumber \\
\le C  \sum_j V( {B}_j)^{r({1\over r}-{1\over
2})}\|\big(\psi F^{(\ell)}\big)(t\sqrt[m]{L})P_{\widetilde B_j}\|_{p\to 2}^r
\| P_{\widetilde B_j}\|_{r\to p}^r\|P_{\widetilde
B_j}f\|_{r}^r
\nonumber \\
 \le C \sup_{x\in X}\big\{V(x,\rho_{\ell})^{r({1\over p}-{1\over
2})}
\|\big(\psi F^{(\ell)}\big)(t\sqrt[m]{L})P_{B(x,\rho_{\ell})}\|^r_{p\to 2}\big\}  \sum_j\|P_{\widetilde
B_j}f\|_{r}^r
\nonumber \\
=  C   \sup_{x\in X}\big\{V(x,\rho_{\ell})^{r({1\over p}-{1\over
2})} \|\big(\psi F^{(\ell)}\big)(t\sqrt[m]{L})P_{B(x,\rho_{\ell})}\|^r_{p\to 2}\big\}  \|f\|_{r}^r.
\end{eqnarray}

\medskip

\noindent
{\it Case 1.}   $ t < 1/(8R_0)$.

\smallskip

We assume that  $\psi \in
C_c(1/16,16)$ so we can  write
$
 \big(\psi F^{(\ell)}\big)(t\sqrt[m]{L})
=\sum_{k=0}^{7}  \big(\chi_{[2^{k-4}, 2^{k-3}) }\psi
F^{(\ell)}\big)(t\sqrt[m]{L}).
$
 If  $t<1/(8R_0)$, then we   use
condition ${\rm (ST^q_{p, 2, m})_{R_0}}$   to show that for every  $\ell \ge 4$,
  \begin{eqnarray}\label{e4.16}
\big\|\big(\psi F^{(\ell)}\big)(t\sqrt[m]{L})P_{B(x, \rho_\ell)}\big\|_{p\to 2 }
&\leq& CV(x,\rho_\ell)^{{1\over 2}-{1\over p}}2^{\ell(1+\varepsilon)n({1\over
p}-{1\over 2})} \sum_{k=0}^{7}  \big\| \delta_{2^{k-3}t^{-1}}\big(\psi
F^{(\ell)}\big) (t\cdot)\big\|_{q}\nonumber\\
&\leq& CV(x,\rho_\ell)^{{1\over 2}-{1\over p}}2^{\ell(1+\varepsilon)n({1\over
p}-{1\over 2})} \| G^{(\ell)} \|_{q}.
 \end{eqnarray}
Note that  by Proposition~\ref{prop2.3},   ${\rm (G_{p,2,m})}$ implies ${\rm (ST^{\infty}_{p, 2, m})}$.
From this,     it can be verified  that for
 $\ell=0, 1,2,3$, $\big\|\big(\psi F^{(\ell)}\big)(t\sqrt[m]{L})P_{B(x, \rho_\ell)}\big\|_{p\to 2 }
 \leq
   CV(x,\rho_\ell)^{{1\over 2}-{1\over p}}  \|F  \|_{q}.$
  Hence
 \begin{eqnarray}\label{e4.17}
\sum_{\ell=0}^{\infty} \sup_{x\in X}\big\{ V(x,\rho_{\ell})^{{1\over p}-{1\over
2}}\big\|\big(\psi F^{(\ell)}\big)(t\sqrt[m]{L})P_{B(x,\rho_{\ell})}\big\|_{p\to
2} \big\}
 &\le & C \|F  \|_{q} + C\sum_{\ell=4}^{\infty}     2^{\ell(1+\varepsilon) n({1\over p}-{1\over 2})}\| G^{(\ell)} \|_{q}
\nonumber\\
 &\leq &  C \|F  \|_{q} + C\|G\|_{B_{q,\, 1}^{ n({1\over p}-{1\over 2}) +\delta}},
\end{eqnarray}
where  $\delta={\varepsilon n }({1\over  p}-{1\over  2})$ and
the last equality follows from definition of Besov space.
See, e.g., \cite[Chap.~VI ]{BL}.
Since $2\delta <\alpha-n({1\over p}-{1\over 2})$,
we have that  $W^{\alpha}_q\subseteq
B_{q, \, 1}^{ n({1/p}-{1/2})+\delta}$ with  $\|G\|_{B_{q,\, 1}^{ n(1/p-1/2)+\delta}}\le C_\alpha
\|G\|_{W^{\alpha}_q}$, see again
\cite{BL}.  However, $\supp F\subseteq [1/4, 4]$ so $ \|G\|_{W^{\alpha}_q}\le
\|F\|_{W^{\alpha}_q}$. Hence the forgoing estimates give
\begin{eqnarray}\label{e4.18}
\sum_{\ell=0}^{\infty}   \sup_{x\in X}\big\{ V(x,\rho_{\ell})^{{1\over p}-{1\over
2}}\big\|\big(\psi F^{(\ell)}\big)(t\sqrt[m]{L})P_{B(x,\rho_{\ell})}\big\|_{p\to
2} \big\}  &\le& C\|F\|_{W^{\alpha}_q}.
\end{eqnarray}

 \medskip

\noindent
{\it Case 2.}   $ t \geq  1/(8R_0)$.

\smallskip

Note that  by Proposition~\ref{prop2.3},   ${\rm (G_{p,2,m})}$ implies ${\rm (ST^{\infty}_{p, 2, m})}$.
At the step \eqref{e4.16}   we use the condition ${\rm (ST^{\infty}_{p, 2, m})}$  in place
of ${\rm (ST^{q}_{p, 2, m})_{R_0}}$, and the similar argument as above shows
\begin{eqnarray}\label{e4.19}
\sum_{\ell=0}^{\infty}   \sup_{x\in X}\big\{ V(x,\rho_{\ell})^{{1\over p}-{1\over
2}}\big\|\big(\psi F^{(\ell)}\big)(t\sqrt[m]{L})P_{B(x,\rho_{\ell})}\big\|_{p\to
2} \big\}  &\le& C\|F\|_{W^{\alpha}_{\infty}}.
\end{eqnarray}.

\medskip

\noindent
{{\em Estimate of   $I\!I$}}.  \
Repeat an argument  leading up to  \eqref{e4.15}, it is easy to see that
\begin{eqnarray*}
 \|\sum_{i,j:\, {d}(x_i,x_j)< 2\rho_{\ell}} P_{\widetilde B_i}\big((1-\psi)
F^{(\ell)}\big) (t\sqrt[m]{L}) P_{\widetilde
B_j}f\|_{r}
&\leq&  C  \|\big((1-\psi) F^{(\ell)}\big)(t\sqrt[m]{L})P_{B(x,\rho_{\ell})}\|_{r\to r}   \|f\|_{r}\\
&\leq & C \|\big((1-\psi) F^{(\ell)}\big)(t\sqrt[m]{L}) \|_{r\to r}   \|f\|_{r},
\end{eqnarray*}
where, for a fixed $N$,  one has the uniform estimates
 \begin{eqnarray*}
\Big|\Big({d\over d\lambda}\Big)^{\kappa} \big((1-\psi) F^{(\ell)}\big)(\lambda)\Big|\leq C_{\kappa}2^{-\ell N}
 (1+|\lambda|)^{-N}\|F\|_{W^{\alpha}_q}.
\end{eqnarray*}
But  (a) of Lemma~\ref{le2.4} then implies that for every $r\in (p, 2)$,
 \begin{eqnarray*}
  \|\big((1-\psi) F^{(\ell)}\big)(t\sqrt[m]{L}) \|_{r\to r} \leq C  2^{-\ell N}\|F\|_{W^{\alpha}_q}.
\end{eqnarray*}
 Therefore,
  \begin{eqnarray}\label{e4.20}
\sum_{\ell=0}^{\infty} \|\big((1-\psi) F^{(\ell)}\big)(t\sqrt[m]{L}) \|_{r\to r} \leq C \|F\|_{W^{\alpha}_q}.
\end{eqnarray}

\medskip

\noindent
{{\em Estimate of  $I\!I\!I$}}.  \ Note that
\begin{eqnarray*}
\big\|\sum_{i,j:\, {d}(x_i,x_j)\geq 2^{\ell(1+\varepsilon)} t}
P_{\widetilde B_i}F^{(\ell)}(t\sqrt[m]{L}) P_{\widetilde
B_j} f \big\|_r^r&=& \sum_{i} \big\|\sum_{j:\, {d}(x_i,x_j)\geq 2^{\ell(1+\varepsilon)} t}
P_{\widetilde B_i}F^{(\ell)}(t\sqrt[m]{L}) P_{\widetilde
B_j} f \big\|_r^r\\
&\leq& \sum_{i}\Big( \sum_{j:\, {d}(x_i,x_j)\geq 2^{\ell(1+\varepsilon)} t} \big\|
P_{\widetilde B_i}F^{(\ell)}(t\sqrt[m]{L}) P_{\widetilde
B_j} f \big\|_r \Big)^r.
\end{eqnarray*}
To go further,   we need the following lemma.

  \bigskip

 \begin{lemma} \label{le4.6}
 Suppose that assumptions of Theorem~\ref{th4.5} are fulfilled.  Let  $r\in (p, 2)$. For
all  $\ell=0, 1, 2, \ldots$  and  all $x_i, x_j$ with  ${d}(x_i,x_j)\geq 2^{\ell(1+\varepsilon)} t$,   there exist
some positive constants $C, c_1, c_2>0$ such that
\begin{eqnarray*}
 \big\|
P_{\widetilde B_i}F^{(\ell)}(t\sqrt[m]{L}) P_{\widetilde
B_j} f\big\|_r\leq
C   e^{-c_1 2^{ \varepsilon \ell m\over m-1}  }
\exp\Big(- c_2\Big({d(x_i, x_j)\over  2^{\ell} t}\Big)^{m\over m-1}\Big) \|F\|_{q} \|P_{\widetilde B_j} f\|_r.
\end{eqnarray*}
\end{lemma}

\begin{proof}  By the formula  \eqref{e4.9},
\begin{eqnarray}
\label{e4.21}
&&\hspace{-1.5cm}
 \big\| P_{\widetilde B_i}F^{(\ell)}(t\sqrt[m]{L}) P_{\widetilde
 B_j} f\big\|_r\nonumber\\
   &\leq&  C
  \|P_{\widetilde B_j} f\|_r \int_{-\infty}^{+\infty}
 |\phi(2^{-\ell}\tau) \hat{G}(\tau)|
 \big\| P_{\widetilde B_i}e^{(i\tau-1)t^mL}  P_{\widetilde B_j}  \big\|_{r\to r}
  \;d\tau,
\end{eqnarray}
where  $G(\lambda)=F(\sqrt[m]\lambda) e^{\lambda}$. Recall that   hypothesis  ${\rm (DG_m)}$  and  ${\rm (G_{p,2,m})}$
  imply ${\rm (GGE_{r,2})}$. This, in combination with Lemma~\ref{le2.1} (with $z=(i\tau -1)t^m$), gives
\begin{eqnarray*}
&&\hspace{-1cm}
 \big\| P_{{\bar B}(x_i, {\rho_{\ell}\over 10})}e^{(i\tau-1)t^mL}  P_{{\bar B}(x_j, {\rho_{\ell}\over 10})}  \big\|_{r\to 2}\\
 &\leq& CV(x_i, {\rho_{\ell}\over 10})^{- ({1\over r}-{1\over 2})}\
 \Big(1+{\rho_\ell\over 10t\sqrt{\tau^2+1}}\Big)^{n({1\over r}-{1\over 2})} \Big(\sqrt{1+\tau^2}\Big)^{n({1\over r}-{1\over 2})}
 \exp\Big(-c \Big({d(x_i,x_j)\over t\sqrt{\tau^2+1}}\Big)^{m\over m-1}\Big),
\end{eqnarray*}
which shows
\begin{eqnarray*}
&&\hspace{-1.5cm}
 \big\| P_{{\bar B}(x_i, {\rho_{\ell}\over 10})}e^{(i\tau-1)t^mL}  P_{{\bar B}(x_j, {\rho_{\ell}\over 10})}  \big\|_{r\to r}\\
 &\leq& \big\| P_{{\bar B}(x_i, {\rho_{\ell}\over 10})}e^{(i\tau-1)t^mL}  P_{{\bar B}(x_j, {\rho_{\ell}\over 10})}
 \big\|_{r\to 2}  \big\| P_{{\bar B}(x_i, {\rho_{\ell}\over 10})}  \big\|_{2\to r}\\
 &\leq& C  \Big(1+{\rho_\ell\over 10t\sqrt{\tau^2+1}}\Big)^{n({1\over r}-{1\over 2})} \Big(\sqrt{1+\tau^2}\Big)^{n({1\over r}-{1\over 2})}
 \exp\Big(-c \Big({d(x_i, x_j)\over t\sqrt{\tau^2+1}}\Big)^{m\over m-1}\Big).
\end{eqnarray*}
Hence, if  $\tau\in [2^{\ell-2}, 2^{\ell}]$, then
\begin{eqnarray*}
  \big\| P_{\widetilde B_i}e^{(i\tau-1)t^mL}  P_{\widetilde B_j}  \big\|_{r\to r}
  &\leq& \big\| P_{{\bar B}(x_i, {\rho_{\ell}\over 10})}e^{(i\tau-1)t^mL}  P_{{\bar B}(x_j, {\rho_{\ell}\over 10})}  \big\|_{r\to r}\\
  &\leq & C 2^{\ell n(1+\varepsilon) ({1\over r }-{1\over 2}) } \exp\Big(-c\Big({d(x_i, x_j)\over  2^{\ell} t}\Big)^{m\over m-1}\Big).
\end{eqnarray*}
Substituting the above inequality  into \eqref{e4.21} and using the fact that $\|{\hat G}\|_{\infty}\leq \|F\|_{q}$ yield  that for
${d}(x_i,x_j)\geq 2^{\ell(1+\varepsilon)} t$,
\begin{eqnarray*}
 \big\|
P_{\widetilde B_i}F^{(\ell)}(t\sqrt[m]{L}) P_{\widetilde
B_j} f\big\|_r
 \leq
C 2^{\ell n(1+\varepsilon) ({1\over r}-{1\over 2}) +1} e^{-c_2 2^{\varepsilon\ell  m\over m-1}  }
\exp\Big(-c_2\Big({d(x_i, x_j)\over  2^{\ell} t}\Big)^{m\over m-1}\Big) \|F\|_{q} \|P_{\widetilde B_j} f\|_r\\
 \leq
C   e^{-c_1  2^{\varepsilon\ell  m\over m-1}  }
\exp\Big(- c_2\Big({d(x_i, x_j)\over  2^{\ell} t}\Big)^{m\over m-1}\Big) \|F\|_{q} \|P_{\widetilde B_j} f\|_r
\end{eqnarray*}
with $c_1=c/4$ and $c_2=c/2$. This proves Lemma~\ref{le4.6}.
\end{proof}

\bigskip

\noindent
 {\em Back to the proof of }{Theorem~\ref{th4.5}.}\  By \eqref{e4.12} for every $i$
\begin{eqnarray*}
 \sum_{ j:\, {d}(x_i,x_j)\geq 2^{\ell(1+\varepsilon)} t}
\exp\Big(-c_2\Big({d({x_i}, x_j)\over 2^{\ell} t}\Big)^{m\over m-1}\Big)
  &\leq & \sum_{k=1}^{\infty}  \sum_{ j:\, 2^k2^{\ell(1+\varepsilon)} t \leq  {d}(x_i,x_j)<2^{k+1} 2^{\ell(1+\varepsilon)} t}
\exp\big(-c_2 {2^{m(k+\ell \varepsilon) \over  m-1 }}\big) \\
&\leq & \sum_{k=1}^{\infty}  2^{2kn}
\exp\big(-c_2 {2^{m(k+\ell \varepsilon) \over  m-1 }}\big)
 \leq   C,
\end{eqnarray*}
which, together with Lemma~\ref{le4.6} and the Cauchy-Schwarz inequality, yields
 \begin{eqnarray*}
&&\hspace{-1.5cm}\big\|\sum_{i,j:\, {d}(x_i,x_j)\geq 2^{\ell(1+\varepsilon)} t}
P_{\widetilde B_i}F^{(\ell)}(t\sqrt[m]{L}) P_{\widetilde
B_j} f \big\|_r^r\\
&\leq & Ce^{- c_1r2^{\varepsilon\ell  m\over m-1}  } \|F\|_{q}^r
\sum_{i} \Big\{ \sum_{j:\, {d}(x_i,x_j)\geq 2^{\ell(1+\varepsilon)} t}
 \exp\Big(- c_2\Big({d(x_i, x_j)\over  2^{\ell} t}\Big)^{m\over m-1}\Big)
  \|P_{\widetilde B_j} f\|_r\Big\}^r \\
	&\leq&  Ce^{- c_1r 2^{\varepsilon\ell  m\over m-1}  } \|F\|_{q}^r
	 \sum_{ j}  \|P_{\widetilde B_j} f\|_r^r
	 \sum_{i:\, {d}(x_i,x_j)\geq 2^{\ell(1+\varepsilon)} t}
\exp\Big(-c_2\Big({d(x_i, x_j)\over 2^{\ell} t}\Big)^{m\over m-1}\Big)  \\
	&\leq&  Ce^{- c_1r 2^{\varepsilon\ell  m\over m-1}  } \|F\|_{q}^r  \| f\|_r^r.
\end{eqnarray*}
 Therefore,
 \begin{eqnarray}\label{e4.22}
 \sum_{\ell=0}^{\infty}    \big\|  \sum_{i,j:\, {d}(x_i,x_j)\geq 2^{\ell(1+\varepsilon)} t}
P_{\widetilde B_i}F^{(\ell)}(t\sqrt[m]{L}) P_{\widetilde
B_j} f\big\|_r &\leq& C \sum_{\ell=0}^{\infty} e^{- c_1  2^{\varepsilon\ell m\over m-1}  } \|F\|_{ q} \|f\|_r\nonumber\\
&\leq& C  \|F\|_{ q} \|f\|_r.
\end{eqnarray}
Estimates \eqref{e4.6} and  \eqref{e4.7} then follow  from  \eqref{e4.11},  \eqref{e4.14},
\eqref{e4.15},   \eqref{e4.18}, \eqref{e4.19}, \eqref{e4.20}  and \eqref{e4.22}.
This completes the proof of Theorem~\ref{th4.5}. \hfill{} $\Box$




\medskip

As we explained in the introduction, a standard application of spectral multiplier theorems is Bochner-Riesz means.
Such application is also a good test to check if the considered multiplier result is sharp.
Let us recall that Bochner-Riesz means of order $\delta$ for a non-negative self-adjoint operator $L$
are defined by the formula
\begin{equation}\label{e4.1}
S_R^{\delta} (L)  = \left(I-{L\over R^m}\right)_+^{\delta},\ \ \ \ R>0.
\end{equation}
\noindent
The case  ${\delta}=0$ corresponds to the spectral projector $E_{\SL}[0, R]$. For
${\delta}>0$ we think of (\ref{e4.1}) as a smoothed version of this
spectral projector; the larger ${\delta}$, the more smoothing.
Bochner-Riesz summability  on $L^p$ describes the range of ${\delta}$
for which $S_R^{\delta} (L)$ are bounded on $L^p$,  uniformly in $R$.

\begin{coro}\label{coro4.3}
Suppose that  the
 operator $L$ satisfies
Davies-Gaffney estimates ${\rm (DG_m)}$  and condition  ${\rm (ST^q_{p, 2, m})}$ with some $1\leq p<2$ and $1\leq q\leq \infty$.

 Then
\begin{eqnarray}\label {e4.2}
\sup_{R>0}\Big\|\Big(I-{L\over R^m}\Big)_+^{\delta}\Big\|_{r\to r}\leq C
\end{eqnarray}
for all  $p<r\leq 2$ and ${\delta}> n(1/r-1/2)-1/q$.
\end{coro}

 \begin{proof} Let $F(\lambda) = (1-\l^m)^{\delta}_+$.  We set
$$
 F(\lambda) =F(\lambda)  \phi(\lambda^m) + F(\lambda)(1-\phi(\lambda^m))=:F_1(\lambda^m) +F_2(\lambda^m),
$$
 where   $\phi\in C^{\infty}(\R)$ is supported in  $ \{ \xi:  |\xi| \geq 1/4 \}$
 and $\phi =1$ for all $|\xi|\geq 1/2$.
Observe  that   $F\in W^{\alpha}_q$ if
and only if ${\delta}>\alpha-1/q$. This, in combination with   Theorem~\ref{th4.5},
shows that
  for all  $p<r\leq 2$ and ${\delta}> n(1/r-1/2)-1/q$,      $ \sup_{R>0} \|F_1(L/R^m)\|_{r\to r}\leq C.$
 On the other hand, it follows from Lemma~\ref{le2.4} that  $\|F_2(L/R^m)\|_{r\to r}\leq C $
 uniformly in $R>0 $. This completes the proof of estimate \eqref{e4.2}.
 \end{proof}

\medskip
\subsection{ Operators with non-empty point spectrum}

It is not difficult to see that   condition   ${\rm (ST^q_{p, 2, m})}$  with  some   $ q< \infty$ implies that
  the set of point spectrum of $L$ is empty.   Indeed, one has for all  $0\leq a< R$ and $x\in X,$
$$
\big\|1\!\!1_{\{a \}  }(\sqrt[m]{L})P_{B(x,r)}\big\|_{p\to 2}
\leq C V(x,r)^{{1\over 2}-{1\over p}} (rR)^{n({1\over p}-{1\over 2})}\big\|1\!\!1_{\{a \} }(R\cdot)\big\|_{q} =0, \ \ \ Rr\ge 1
$$
and therefore $1\!\!1_{\{a\} }(\sqrt[m]{L})=0$ so the point spectrum of $L$
is empty, see \cite{DOS}.
In particular, ${\rm (ST^q_{p, 2, m})}$ cannot hold for any $q<\infty$ for elliptic
operators on compact manifolds or for the harmonic oscillator.
To be able to  study these operators as well,  we introduce a
variation of  condition ${\rm (ST^q_{p, 2, m})}$. Following  \cite{CowS, DOS}, for an even Borel function $F$
with  $\supp F\subseteq [-1, 1]$ we define
the norm $\|F\|_{N,q}$ by
$$
\|F\|_{N,q}=\left({1\over 2N}\sum_{\ell=1-N}^{N} \sup_{\lambda\in
 [{\ell-1\over N}, {\ell\over N})} |F(\lambda)|^q\right)^{1/q},
 $$
 where $q\in [1, \infty)$ and $N\in \NN$. For $q=\infty$, we put
 $\|F\|_{N, \infty}=\|F\|_{{\infty}}$.
 It is obvious that $\|F\|_{N,q}$ increases monotonically in $q$.

     Consider a non-negative self-adjoint operator $L$  and numbers
$p$ and $q$ such that $1\leq p< 2$ and $1\leq
q\leq\infty$.   Following \cite{COSY}, we shall  say  that   $L$
  satisfies the  {\it Sogge spectral cluster condition}  if: for  a fixed natural number $\kappa$ and
 for all $N\in \NN$ and  all even
  Borel functions $F$  such that\, $\supp F\subseteq [-N, N]$,
$$
\big\|F(\SL)P_{B(x, r)} \big\|_{p\to 2} \leq
 CV(x,r)^{{1\over 2}-{1\over p}}(Nr)^{n({1\over p}-{1\over 2})}\|\delta_NF\|_{N^\kappa,\, q}
\leqno{\rm (SC^{q, \kappa}_{p, 2, m})}
$$
for  all $x\in X$ and $r\geq 1/N$.  For $q=\infty$, ${\rm (SC^{\infty,\kappa}_{p, 2, m})}$ is independent of $\kappa$  so we write
it as ${\rm (SC^{\infty}_{p, 2, m})}$.

\medskip

\begin{remark}\label{re4.7}\, It is easy to check that for $\kappa\geq 1$, ${\rm (SC^{q, \kappa}_{p, 2, m})}$ implies
${\rm (SC^{q,1}_{p, 2, m})}$. Moreover, if  $\mu(X)<\infty$, then
 conditions ${\rm (SC^{\infty}_{p, 2, m})}$  and ${\rm (ST^{\infty}_{p, 2, m})}$
are  equivalent(see  Proposition 3.11, \cite{COSY}).
\end{remark}

The next theorem
 is a version of Theorem~\ref{th4.5} suitable for the operators satisfying
 condition ${\rm (SC^{q, \kappa}_{p, 2, m})}$.

\begin{theorem}\label{th4.11} Suppose the operator $L$ satisfies
Davies-Gaffney estimates ${\rm (DG_m)}$,
 conditions   ${\rm (G_{p,2,m})}$
and condition  ${\rm (SC^{q, \kappa}_{p, 2, m})}$ for a fixed
$\kappa \in {\mathbb N}$ and some $p, q$ such that $1\leq p<2
$ and $1\leq q\leq \infty$. In addition, we assume that for
any $\varepsilon>0$ there exists a constant $C_\varepsilon$ such
that for all $N\in {\mathbb N}$ and all even Borel functions $F$
such that supp $F\subset [-N,N]$,
$$
\|F(\SL)\|_{p\to p}\leq C_\varepsilon N^{\kappa n({1\over p}-{1\over
2})+\varepsilon}\|\delta_N F\|_{N^{\kappa},q}. \leqno{\rm (AB_{p, m})}
$$
Let $p<r\leq 2$. Then  for any  function $F$ such that $\supp F \subseteq [1/4, 4]$ and $\|F\|_{W^{\alpha}_q}<\infty$
 for some $\alpha>\max\{n(1/p-1/2),1/q\}$, the operator $F(t \sqrt[m]{L})$ is bounded on $L^p(X)$ for all $t>0.$  In addition,
 \begin{equation}
\label{e4.23} \sup_{t>0}\|F(t \sqrt[m]{L})\|_{r\to r} \leq
C\|F\|_{W^{\alpha}_q}.
\end{equation}
\end{theorem}

\medskip

Note that condition ${\rm (SC^{q, \kappa}_{p, 2, m})}$ is weaker than ${\rm (ST^q_{p, 2, m})}$ and we need  a priori estimate
  ${\rm (AB_{p, m})}$  in Theorem~\ref{th4.11}.
See also \cite[Theorem 3.6]{CowS} and \cite[Theorem 3.2]{DOS} for related results. Once
  ${\rm (SC^{q, \kappa}_{p, 2, m})}$ is proved, a priori estimate
  ${\rm (AB_{p, m})}$  is not difficult to check in general.

\begin{proposition}\label{prop4.9} Suppose that $\mu(X)<\infty$ and ${\rm (SC^{q,1}_{p, 2, m})}$  for some $p, q$ such that $1\leq p<2
$ and $1\leq q\leq \infty$.
Then
\begin{eqnarray*}
\|F(\SL)\|_{p\to p} \leq  C  N^{n({1\over p}-{1\over 2}) }\|\delta_N
F\|_{N,q}
\end{eqnarray*}
for all $N\in \NN$ and all Borel  functions $F$ such that supp $F\subseteq [-N, N]$.
\end{proposition}

\begin{proof}   We follow  Proposition 3.7 of \cite{DOS}
to  prove  the  result (see also \cite{DOS}). Since  $\mu(X)<\infty$, we may assume that $X=B(x_0, 1)$ for some $x_0\in X$.
It  follows from H\"older's inequality and condition ${\rm (SC^{q,1}_{p, 2})}$
 that
\begin{eqnarray*}
\|F(\SL)\|_{p\to p}&\leq& V(X)^{{1\over p}-{1\over 2}}
\|F(\SL)P_{B(x_0, 1)}\|_{p\to 2}\\
&\leq & C V(X)^{{1\over p}-{1\over 2}}V(X)^{{1\over 2}-{1\over
p}}N^{n({1\over p}-{1\over 2})}\|\delta_NF\|_{N,\, q}\\
&\leq & C  N^{n({1\over p}-{1\over 2})}\|\delta_N
F\|_{N,q}.
\end{eqnarray*}
This means that ${\rm (AB_{p, m})}$ is satisfied.
This proves Proposition~\ref{prop4.9}.
\end{proof}

\medskip

\begin{remark}\label{re4.10}
Suppose that $\mu(X)<\infty$   and ${\rm (SC^{q, \kappa}_{p, 2, m})}$ holds for some $\kappa\geq 1$. Then
${\rm (SC^{\infty}_{p, 2, m})}$ and ${\rm (G_{p,2,m})}$ are satisfied by  Remark~\ref{re4.7}.
In addition,   ${\rm (AB_{p, m})}$ holds
by   Proposition~\ref{prop4.9}.  Therefore,   Theorem~\ref{th4.11} holds in this case without assumptions
${\rm (G_{p,2,m})}$ and ${\rm (AB_{p, m})}$.
\end{remark}

 \medskip

\noindent
{\it Proof of Theorem~\ref{th4.11}.}  We consider two cases: $t\geq 1/4$ and $t\leq 1/4.$

\smallskip
\noindent
{\it Case} (1).  $t\geq 1/4$.

\smallskip
If  $t\geq 1/4$ then  $\supp \delta_t F
 \subset [0,16]$. By ${\rm (AB_{p, m})}$,
 \begin{eqnarray*}
\|F(t\SL)\|_{p \to p} &\leq&C 16^{\kappa n({1\over p}-{1\over
2})+\varepsilon}
 \|\delta_{16}(F(t\cdot))\|_{16^\kappa,q}
\leq C  \|F\|_{\infty}.
\end{eqnarray*}
 Recall that if  $\alpha>1/q$,
  then $W^{\alpha}_q(\RR) \subseteq L^{\infty}(\RR)\cap C(\RR)$ and $\|F\|_{\infty}\leq C\|F\|_{W^{\alpha}_q}$.
Hence
$$
 \|F(t\SL)\|_{p\to p} + \|F(t\SL)\|_{2\to 2}\leq C\|F\|_{\infty}\leq
C\|F\|_{W^{\alpha}_q}.
$$
Now   \eqref{e4.23} follows   by interpolation.

\smallskip
\noindent
{\it Case} (2).  $t\leq 1/4$.

\smallskip
Let $\xi\in C_c^\infty$ be an even function   such that $\supp
\xi\subset [-1/2, 1/2],\, \hat{\xi}(0)=1$ and $\hat{\xi}^{(k)}(0)=0$ for
all $1\leq k\leq [\alpha]+2$. Write
$\xi_{N^{\kappa-1}}=N^{\kappa-1}\xi(N^{\kappa-1}\cdot)$ where
$N=8[t^{-1}] +1$ and $[t^{-1}]$ denotes the integer part of
$t^{-1}.$ Following \cite{CowS} we write
$$
F(t\SL)=\big(\delta_tF - \xi_{N^{\kappa-1}}\ast
\delta_tF\big)(\SL)+(\xi_{N^{\kappa-1}}\ast \delta_tF)(\SL).
$$
Now we prove that
\begin{eqnarray}
\|\big(\delta_tF - \xi_{N^{\kappa-1}}\ast
\delta_tF\big)(\SL)\|_{r\to r}&\leq& C \|F\|_{W^{\alpha}_q}.
\label{e4.24}
\end{eqnarray}

Observe that $\supp  (\delta_tF - \xi_{N^{\kappa-1}}\ast \delta_tF
)\subseteq [-N, N]$. We apply  ${\rm (AB_{p, m})}$   to obtain
\begin{eqnarray}\label{e4.25} \hspace{1cm}
 \|\big(\delta_tF - \xi_{N^{\kappa-1}}\ast \delta_tF\big)(\SL)\|_{p \to p}
  \leq C   N^{\kappa n({1\over
p}-{1\over 2})+\varepsilon}\big\|\delta_{N}\big(\delta_tF -
\xi_{N^{\kappa-1}}\ast \delta_tF\big)\big\|_{N^\kappa,q}.
\end{eqnarray}
Everything then boils down to estimate $\|\cdot\|_{N^\kappa,q}$
norm of $\delta_{N}\big(\delta_tF - \xi_{N^{\kappa-1}}\ast
\delta_tF\big).$ We make the following claim.
  For the proof we referee readers to  \cite[(3.29)]{CowS} or \cite[Propostion 4.6]{DOS}.

\begin{lemma}\label{le4.12}
Suppose that\ \ $\xi\in C_c^\infty$ is an even function such that
$\supp \xi\subset [-1/2, 1/2],  \ \ \hat{\xi}(0)=1$ and
$\hat{\xi}^{(k)}(0)=0$ for all $1\leq k\leq [\alpha]+2$. Next assume that
 $\supp H\subset [-1,1]$. Then
\begin{eqnarray}\label {e4.26}
\|H-\xi_N\ast H\|_{N,q}\leq CN^{-\alpha}\|H\|_{W^{\alpha}_q}
\end{eqnarray}
for all $\alpha>1/q$ and any $N\in \NN$.
\end{lemma}

Note that  $\delta_N\big(\delta_tF - \xi_{N^{\kappa-1}}\ast \delta_tF\big)=\delta_{Nt}F - \xi_{N^{\kappa}}\ast \delta_{Nt}F$.
Now, if $\alpha>\max \{{n({1/p}-{1/2})}, 1/q\}$
then \eqref{e4.24} follows from Lemma~\ref{le4.12}, estimate \eqref{e4.25} and the interpolation theorem.

\medskip

It remains to show that
\begin{eqnarray}
\|(\xi_{N^{\kappa-1}}\ast \delta_tF)(\SL)\|_{r\to r}&\leq& C
\|F\|_{W^{\alpha}_q}. \label {e4.27}
\end{eqnarray}
Let $ F^{(\ell)} $ be functions defined in
(\ref{e4.9}). we can write
$$
(\xi_{N^{\kappa-1}}\ast \delta_tF)(\lambda)=\sum_{\ell \ge 0}
\big(\xi_{N^{\kappa-1}}\ast \delta_tF^{(\ell)}\big)(\lambda),
$$
 and hence
\begin{eqnarray}\label{e4.28}
\big\|\big(\xi_{N^{\kappa-1}}\ast \delta_tF \big)(\SL)\big\|_{r\to r
} &\le&
\sum_{\ell \ge 0}\big\|\big(\xi_{N^{\kappa-1}}\ast \delta_tF^{(\ell)}\big) (\SL) \big\|_{r\to r}.
\end{eqnarray}

As in the proof of
Theorem~\ref{th4.5}, we  fix  an $ \epsilon>0$  such that
 $
 2n\epsilon ({1/p}-{1/2}) \leq  \alpha-n({1/p}-{1/2}).
 $
For every $t>0$, and every  $\ell$, we let $\rho_{\ell}=2^{\ell(1+\epsilon)}t$.
Let  $\psi \in C_c^\infty(1/16, 16)$ such that $\psi(\lambda)=1$ for $\lambda \in (1/8,8)$.
We    decompose
\begin{eqnarray}\label{e4.29}
\big(\xi_{N^{\kappa-1}}\ast \delta_tF^{(\ell)} \big)(\sqrt[m]{L})f
&=& \sum_{i,j:\, {d}(x_i,x_j)< 2\rho_{\ell}} P_{\widetilde B_i}
\big(\delta_t\psi \big(\xi_{N^{\kappa-1}}\ast \delta_tF^{(\ell)} \big)\big) (\sqrt[m]{L}) P_{\widetilde
B_j}f\nonumber\\
  &&+ \sum_{i,j:\, {d}(x_i,x_j)< 2\rho_{\ell}} P_{\widetilde B_i}\big((1-\delta_t\psi)
\big(\xi_{N^{\kappa-1}}\ast \delta_tF^{(\ell)} \big)\big) (\sqrt[m]{L}) P_{\widetilde
B_j}f\nonumber\\
&&+  \sum_{i,j:\, {d}(x_i,x_j)\geq 2\rho_{\ell}} P_{\widetilde B_i}\big(\xi_{N^{\kappa-1}}
\ast \delta_tF^{(\ell)} \big)(\sqrt[m]{L}) P_{\widetilde
B_j}f\nonumber\\
&=&: {\mathscr F}^{(\ell)}_{1}(f) +{\mathscr F}^{(\ell)}_{2}(f)  +{\mathscr F}^{(\ell)}_{3}(f) .
\end{eqnarray}
We shall show that
 \begin{eqnarray}\label{e4.30}
 \sum_{\ell=0}^{\infty}\big\|{\mathscr F}^{(\ell)}_{i}(f)
\big\|_{r\to r}
 \le     C\|F\|_{W^{\alpha}_q}, \ \ \  \ i=1,2, 3.
\end{eqnarray}

Similar argument as in  \eqref{e4.15} above give
\begin{eqnarray*}
\big\|{\mathscr F}^{(\ell)}_{1}(f)
\big\|_{r\to r
}
&\le&    \sup_{x\in X} \big\{
V(x,\rho_\ell)^{{1\over p}-{1\over
2}} \big\|\big(\delta_t\psi  (\xi_{N^{\kappa-1}}\ast
\delta_tF^{(\ell)})\big) (\SL)
 P_{B(x,\rho_\ell)}\big\|_{p\to 2 }\big\}\|f\|_r.
\end{eqnarray*}
We  then follow the proof of Theorem 3.6 of \cite{COSY} to get
\begin{eqnarray*}
\big\|\big(\delta_t\psi  (\xi_{N^{\kappa-1}}\ast
\delta_tF^{(\ell)})\big) (\SL)
 P_{B(x,\rho_\ell)}\big\|_{p\to 2 } \le CV(x,\rho_\ell)^{{1\over 2}-{1\over p}}2^{\ell(1+\epsilon) n({1\over
p}-{1\over 2})}   \|F^{(\ell)}\|_{q}.
\end{eqnarray*}
This shows  \eqref{e4.30} for $i=1$ (see  \eqref{e4.18} above).

\smallskip
For $i=2$, the proof of \eqref{e4.30}  is similar to that of \eqref{e4.20}.
For   $i=3$,  we write
\begin{eqnarray*}
\big\|{\mathscr F}^{(\ell)}_{3}(f) \big\|_r^r
&\leq& \sum_{i}\Big( \sum_{j:\, {d}(x_i,x_j)\geq 2^{\ell(1+\epsilon)} t} \big\|
P_{\widetilde B_i}\big(\xi_{N^{\kappa-1}}\ast \delta_tF^{(\ell)} \big) (\sqrt[m]{L}) P_{\widetilde
B_j} f \big\|_r \Big)^r.
\end{eqnarray*}
Observe that if
${d}(x_i,x_j)\geq 2^{\ell(1+\epsilon)} t$,  then  by Lemma~\ref{le4.6},
\begin{eqnarray*}
 \big\| P_{\widetilde B_i}\big(\xi_{N^{\kappa-1}}\ast \delta_tF^{(\ell)} \big)(\sqrt[m]{L}) P_{\widetilde
 B_j} f\big\|_r\nonumber\\
   \leq   C
  \|P_{\widetilde B_j} f\|_r \int_{-\infty}^{+\infty} |{\widehat{\xi_{N^{\kappa-1}}}}(\tau)|
 |\phi(2^{-\ell}\tau) \hat{G}(\tau)|
 \big\| P_{\widetilde B_i}e^{(i\tau-1)t^mL}  P_{\widetilde B_j}  \big\|_{r\to r}
  \;d\tau\\
 \leq
C  e^{-c_1 2^{ \epsilon \ell m\over m-1}  }
\exp\Big(- c_2\Big({d(x_i, x_j)\over  2^{\ell} t}\Big)^{m\over m-1}\Big)   \|F\|_{q}  \|P_{\widetilde B_j} f\|_r.
\end{eqnarray*}
The rest of the proof of \eqref{e4.30} for $i=3$ is just a repetition of the proof  of   \eqref{e4.22} so
 we skip it here.  This completes  the proof of Theorem~\ref{th4.11}.
 \hfill{} $\Box$

\bigskip

\section{H\"ormander-type spectral multiplier theorems}\label{sec5}
\setcounter{equation}{0}

The aim of this section is to obtain H\"ormander-type spectral multiplier theorems
to include singular integral versions of Theorems~\ref{th4.5} and  ~\ref{th4.11}.
 We continue with the assumption that $(X, d, \mu)$ is  a  metric measure  space satisfying the doubling property
  and recall that $n$ is the  doubling dimension from  condition \eqref{e1.3}.
 Fix a non-trival auxiliary function $\eta\in C_c^{\infty}(0, \infty)$.

 \begin{theorem}\label{th4.1}
Let  $L$  be  a non-negative self-adjoint operator   on $L^2(X)$ satisfying
Davies-Gaffney estimates ${\rm (DG_m)}$  and condition  ${\rm (ST^q_{p, 2, m})}$ for some $p, q$ satisfying
  $1\leq p<2$ and $1\leq  q\leq \infty$. Then for any  bounded Borel
function $F$ such that
$\sup_{t>0}\|\eta\, \delta_tF\|_{W^{\alpha}_q}<\infty $ for some
$\alpha>\max\{n(1/p-1/2),1/q\}$,  the operator
$F(L)$ is bounded on $L^r(X)$ for all $p<r<p'$.
In addition,
\begin{eqnarray*}
   \|F(L)  \|_{r\to r}\leq    C_\alpha\Big(\sup_{t>0}\|\eta\, \delta_tF\|_{W^{\alpha}_q}
   + |F(0)|\Big).
\end{eqnarray*}
\end{theorem}

\begin{proof} Note that   by Proposition~\ref{prop2.3},
  ${\rm (ST^{q}_{p, 2, m})} \Rightarrow {\rm (ST^{\infty}_{p, 2, m})} \Rightarrow {\rm (G_{p,2,m})}$.
  Now Theorem~\ref{th4.1} follows from  Theorems~\ref{th3.3} and  ~\ref{th4.5}.
\end{proof}

\medskip

Note that Gaussian bounds ${\rm (GE_m)}$ implies estimates ${\rm (DG_m)}$ and ${\rm (G_{p, 2, m})}$
so by Proposition~\ref{prop2.3}, ${\rm (ST_{p, 2, m}^{\infty})}$ holds
for $q=\infty$. This means that one can omit
conditions ${\rm (DG_m)}$  and  ${\rm (ST^q_{p, 2, m})}$ in Theorem~\ref{th4.1}
  if the case $q=\infty$ is considered. We describe the details in Theorem~\ref{th4.2} below.

 \medskip

 \begin{theorem} \label{th4.2} Let  $L$  be  a non-negative self-adjoint operator   on $L^2(X)$ satisfying
Gaussian  estimates ${\rm (GE_m)}$. Let  $1\leq p<2$. 

Then for any  bounded Borel
function $F$ such that
$\sup_{t>0}\|\eta\, \delta_tF\|_{W^{\alpha}_{\infty}}<\infty $ for    some
$\alpha> n(1/p-1/2)  $ the operator
$F(L)$ is bounded on $L^r(X)$ for all $p<r<p'$.
In addition,
\begin{eqnarray*}
   \|F(L)  \|_{r\to r}\leq    C_\alpha\Big(\sup_{t>0}\|\eta\, \delta_tF\|_{W^{\alpha}_{\infty}}
   + |F(0)|\Big).
\end{eqnarray*}
 \end{theorem}

 \begin{proof}  Theorem~\ref{th4.2} follows Proposition~\ref{prop2.3} and  Theorem~\ref{th4.1}.
\end{proof}

The next theorem
 is a variation of Theorem~\ref{th4.1} suitable for the operators satisfying
 condition ${\rm (SC^{q, \kappa}_{p, 2, m})}$.

\begin{theorem}\label{th4.8}  Suppose the operator $L$ satisfies
Davies-Gaffney estimates ${\rm (DG_m)}$,
 conditions   ${\rm (G_{p,2,m})}$ and   ${\rm (SC^{q, \kappa}_{p, 2, m})}$ for some $p,q$ such that $1\leq p<2$
and $1\leq q\leq \infty$, and  a fixed natural number $\kappa$. In
addition, we assume that for any $\varepsilon>0$ there exists a
constant $C_\varepsilon$ such that for all $N\in \NN$ and
all even Borel functions $F$ such that supp $F\subset [-N,N]$,
$$
\|F(\SL)\|_{p\to p}\leq C_\varepsilon N^{\kappa n({1\over p}-{1\over
2})+\varepsilon}\|\delta_N F\|_{N^{\kappa},q}. \leqno{\rm (AB_{p, m})}
$$
 
  Then for any even bounded Borel function $F$
 such that
$\sup_{t>0}\|\eta\, \delta_tF\|_{W^{\alpha}_q}<\infty $ for some
$\alpha>\max\{n(1/p-1/2),1/q\}$  the operator $F(L)$ is bounded on
$L^r(X)$ for all $p<r<p'$. In addition,
\begin{eqnarray*}
   \|F(L)  \|_{r\to r}\leq    C_\alpha\Big(\sup_{t>0}\|\eta\, \delta_tF\|_{W^{\alpha}_q}
   + |F(0)|\Big).
\end{eqnarray*}
\end{theorem}

 \begin{proof}  Theorem~\ref{th4.8} follows  Theorems~\ref{th3.3} and  ~\ref{th4.11}.
\end{proof}

\begin{remark}\label{re4.1000}
Suppose that $\mu(X)<\infty$   and ${\rm (SC^{q, \kappa}_{p, 2, m})}$ holds for some $\kappa\geq 1$. Then
${\rm (SC^{\infty}_{p, 2, m})}$ and ${\rm (G_{p,2,m})}$ are satisfied by  Remark~\ref{re4.7}.
In addition,   ${\rm (AB_{p, m})}$ holds
by   Proposition~\ref{prop4.9}.  Therefore,   Theorem~\ref{th4.8} holds in this case without assumptions
${\rm (G_{p,2,m})}$ and ${\rm (AB_{p, m})}$.
\end{remark}

\medskip

 \section{Applications}\label{appl5}
 \setcounter{equation}{0}

 As an illustration of our  results we shall  discuss a few of possible applications.
Our main results,   Theorems~\ref{th4.5}, ~\ref{th4.11}, ~\ref{th4.1} and ~\ref{th4.8},
   can be applied to all examples which are
  discussed in \cite{DOS} and \cite{COSY}. Those  include the standard Laplace operator; Laplace-Beltrami operator acting
on compact manifolds; the Laplace-Beltrami operator and some of its perturbation on
asymptotically conic manifolds, see \cite{GHS}; the harmonic oscillator and its perturbations; homogeneous
sub-Laplacians on nilpotent Lie groups. We do not discuss the details here as the obtained corollaries coincide
with applications described in \cite{COSY}, except that we are not able to prove endpoints estimates for
Bochner-Riesz sumability.
We suspect that endpoints results possibly do not hold in $m$-th order operators setting.

\medskip
\subsection{$m$-th order differential operators on compact manifolds}
For a general positive definite elliptic operator on a compact manifold, condition ${\rm (GE_m)}$
holds by general elliptic regularity theory.  As a consequence of
Theorem~\ref{th4.8}, we obtain alternative proof
of Theorem 3.2 in \cite{SS} described by A. Seeger and C. Sogge. The result can be stated in the following way.

\begin{theorem}\label{th5.1} Let $M$ be a compact connected manifold without boundary of dimension
  $n\geq 2$.
Let $P_m(x, D)$ be a positive definite elliptic pseudo-differential operator of order $m$ on
  $M$.  Suppose that  for each $x\in M$, the cosphere
\begin{eqnarray}\label{e5.1}
\Sigma_x=\{\xi\in T^{\ast}_x M\backslash 0: P_m(x,\xi)=1 \}
\end{eqnarray}
 has nonzero Gaussian curvature everywhere, where
$P_m(x, \xi)$ is the principal symbol.
 Let $1\leq p \leq 2(n+1)/(n+3)$  and $1\leq q\leq \infty$.
Then for any even bounded Borel function $F$
 such that
$\sup_{t>0}\|\eta\, \delta_tF\|_{W^{\alpha}_q}<\infty $ for some
$\alpha>\max\{n(1/p-1/2),1/q\}$,  the operator $F( {P_m})$ is bounded on
$L^r(X)$ for all $p<r<p'$. In addition,
\begin{eqnarray*}
   \|F( {P_m})  \|_{r\to r}\leq    C_\alpha \sup_{t>0}\|\eta\, \delta_tF\|_{W^{\alpha}_q}.
\end{eqnarray*}
\end{theorem}

\noindent
\begin{proof} Under the non-degenerate  assumption of the cospheres
$\Sigma_x$, it follows by Corollary 2.2 of \cite{SS2} that estimates ${\rm (SC^{2,1}_{p,2})}$ hold  for
$1\leq p \leq 2(n+1)/(n+3)$. Then the result is a consequence of  Theorem~\ref{th4.8} and Remark~\ref{re4.1000}.
\end{proof}

\medskip
\subsection{  $m$-th order  elliptic differential operators with constant coefficients}
Let $P_m(D)$ be a real homogeneous   elliptic polynomial of order $m$ on $\R^n, n\geq 2$, and $\Sigma$
is a hypersurface
defined by
\begin{eqnarray}\label{e5.2}
\Sigma=\{\xi\in \R^n: |P_m(\xi)|=1 \},
\end{eqnarray}
where
$P_m(\xi)$ is the   symbol.  Recall that $\Sigma$ is of finite type if there exist $k\in{\mathbb N}$ and $C>0$
such that
\begin{eqnarray}\label{e5.3}
\sum_{j=1}^k\big|\langle \eta, \nabla\rangle^{j} P_m(\xi)\big| \geq C>0, \ \ \xi\in\Sigma \ \
{\rm and}\ \ \eta\in {\bf S}^{n-1},
\end{eqnarray}
where $\langle \eta, \nabla\rangle=\sum_{i=1}^n \eta_i {\partial/\partial x_i}.$
The least $k$ in \eqref{e5.3} is called the type order of $\Sigma.$ Say that $\Sigma$
is convex if
\begin{eqnarray}\label{e5.4}
\Sigma\subseteq \big\{\eta\in\R^n \big| \  \langle \eta-\xi, \nabla P_m(\xi)\rangle\geq 0\big\}, \ \ \xi\in\Sigma
\end{eqnarray}
or
\begin{eqnarray}\label{e5.5}
\Sigma\subseteq \big\{\eta\in\R^n \big| \  \langle \eta-\xi, \nabla P_m(\xi)\rangle\leq 0\big\}, \ \ \xi\in\Sigma.
\end{eqnarray}

For a given $P_m$, we know that the corresponding $\Sigma$ is always of finite type and $2\leq k\leq m$.
But it is obviously not always convex. The hypersurface $\Sigma$ is convex and $k=2$ if and only if
$\Sigma$ has nonzero Gaussian curvature everywhere.
A simple example of polynomials whose level hypersurface $\Sigma$ is of type $m$ is
$\xi_1^m +\cdots +\xi_n^m (m=4,6, \cdots)$. We notice that there exist polynomials $P_m$ whose level
hypersurfaces  $\Sigma$ are of type $k (<m)$.     For example, when
$P_6(\xi)=\xi_1^6 +5\xi_1^2\xi_2^4 +\xi_2^6$,
the
corresponding hypersurface $\Sigma$ is of type $4$, but $m=6$ (see \cite{DYa}).

\begin{proposition}\label{prop5.2}
Let $P_m(D)$ be a real homogeneous   elliptic polynomial of order $m$ on $\R^n, n\geq 2$.
Suppose $\Sigma$ is a convex hypersurface of finite
type $k $  for $2\leq k\leq m$  and that  $1\le p \leq {2(n-1+k)/(n-1+2k)}$. Alternatively
assume that $\Sigma$ has nonzero Gaussian curvature everywhere and that $1\leq p \leq 2(n+1)/(n+3)$.
Then we have
 \begin{equation}\label{e5.6}
 \|d E_{\sqrt[m]{L}}(\lambda)\|_{p \to p' }\le C\lambda^{n(\frac{1}{p} -\frac{1}{p'}) -1}, \ \ \ \ \lambda>0.
 \end{equation}
Hence condition ${\rm (ST^{2 }_{p, 2, m})} $ holds.
\end{proposition}

\begin{proof}
 Estimates  \eqref{e5.6} and ${\rm (ST^{2 }_{p, 2, m})} $ follow from
Theorem B in \cite{BNW} and Theorem 1 in \cite{G}.
\end{proof}

\medskip

  We are now able to state the following result describing spectral multipliers for $m$-th
   order  elliptic differential operators with constant coefficients.

\begin{theorem}\label{th5.3} Suppose $\Sigma$ is a convex hypersurface of finite
type $k $  for $2\leq k\leq m$  and that  $1\le p \leq {2(n-1+k)}/{(n-1+2k)}$. Alternatively
assume that $\Sigma$  has nonzero Gaussian curvature everywhere and that $1\leq p \leq 2(n+1)/(n+3)$.
Then for any even bounded Borel function $F$
 such that
$\sup_{t>0}\|\eta\, \delta_tF\|_{{W^{\alpha}_q}}<\infty $ for some
$\alpha>n(1/p-1/2)$ and $1\leq q\leq \infty$, the operator $F( {P_m})$ is bounded on
$L^r(X)$ for all $p<r<p'$. In addition,
\begin{eqnarray*}
   \|F( {P_m})  \|_{r\to r}\leq    C_\alpha \sup_{t>0}\|\eta\, \delta_tF\|_{W^{\alpha}_q}.
\end{eqnarray*}
\end{theorem}

\begin{proof} It is known that  the
semigroup $e^{-tL}$ has integral kernels $p_t(x,y)$ satisfying the following estimates
${\rm (GE_m)}$ (see \cite{Da}). Now Theorem \ref{th5.3}
is a straightforward consequence of Proposition \ref{prop5.2} and Theorem \ref{th4.1}.
\end{proof}

\medskip
\subsection{Biharmonic operators with rough potentials}\label{appl5.3} In the section we
consider the biharmonic operator $\Delta^2=(\partial_{1}^2   +\partial_{2}^2+\partial_{3}^2)^2$
acting on $L^2(\R^3)$.
Assume now that $n=3$ and $V$ is a real-valued
measurable function such that $
0\leq V\in L^1({\mathbb R}^3)\cap L^2({\mathbb R}^3).
$
We define a self-adjoint operator $L$ as Friedrich's extension of the operator $\Delta^2+V$ initially defined on
$C_c^\infty(\R^3)$.

To be able to apply our results to the operator $L$ we first show that the corresponding semigroup
satisfies Davies-Gaffney estimates  ${\rm (DG_4)}$ and $4$-th order Gaussian bounds ${\rm (GE_4)}$.

\begin{proposition}\label{prop5.4}
Suppose that $
0\leq V\in L^1({\mathbb R}^3)$. Then the semigroup generated by the operator $H =\Delta^2+V$
and the corresponding heat kernel $p_t(x,y)$ satisfies
Davies-Gaffney estimates  ${\rm (DG_4)}$ and Gaussian bounds ${\rm (GE_4)}$.
\end{proposition}

\begin{proof}
Following \cite{Du} we consider the set of linear functions $\psi \colon \R^3 \to \R$ of the form
$\psi(x)=a \cdot x$, where $a=(a_1,a_2,a_3) \in {\bf S}^2$. Then for $\lambda \in \R$ we consider perturbed operator
$$
H_{\lambda \psi}=e^{-\lambda \psi}He^{\lambda \psi} = e^{-\lambda \psi}\Delta^2e^{\lambda \psi}+V
=\Delta_{\lambda \psi}^2 +V,
$$
where $\Delta_{\lambda \psi}=  e^{-\lambda \psi}\Delta e^{\lambda \psi}=
(\partial_{1}+a_1 \lambda)^2   +(\partial_{2}+a_2 \lambda)^2+(\partial_{3}+a_3 \lambda)^2   $, see
 \cite[Lemma 10]{Du}.  Note that
$$
\exp(-t H_{\lambda \psi})=e^{-\lambda \psi}\exp(-t H)e^{\lambda \psi}.
$$
By Lemma 10 of \cite{Du} there exists a constant $c>0$ such that
$$
\| \exp(-t\Delta_{\lambda \psi}^2)\|_{2\to 2} =e^{c\lambda^4t}.
$$
However  we assume that $V\ge 0$ so 
$$
\|e^{-\lambda \psi}\exp(-t H)e^{\lambda \psi}\|_{2\to 2} \le e^{c\lambda^4t},
$$
see also \cite{Da}.
Now consider  $a=(a_1,a_2,a_3) \in {\bf S}^2$ such that $\psi(x)-\psi(y)=|x-y|$.
Then
$$
\big\|P_{B(x, t^{1/4})} e^{-tH} P_{B(y, t^{1/4})}\big\|_{2\to {2}}
\le e^{c\lambda^4t-\lambda(|x-y|-2 t^{1/4})}.
$$
Taking infimum over $\lambda$ in the above
inequality proves estimates  ${\rm (DG_4)}$.

\bigskip

To prove Gaussian estimates ${\rm (GE_4)}$ we first note that
$\|(I+t\Delta^2)^{-1/2}\|_{2\to {\infty}}\le C t^{-3/4}$.
However we assume that $V(x) \ge 0$ for all $x \in \R^3$ so
$$
\langle (I+t(\Delta^2+V))f, f \rangle \ge \langle (I+t \Delta^2 )  f, f \rangle=\|(I+t\Delta^2)^{-1/2}f\|_2.
$$
Hence $\|(I+tH)^{-1/2}\|_{2\to {\infty}}\le C t^{-3/4}$ and
$$
\|\exp(-t H)\|_{2\to \infty }\le \|(I+tH)^{-1/2}\|_{2\to {\infty}}\|(I+tH)^{-1/2}\exp(-t H)\|_{2\to 2 }\le
C t^{-3/4}.
$$
This proves on-diagonal bounds for the corresponding heat kernel. Now it  suffices to ensure that all
assumptions of our abstract results hold. To prove off-diagonal Gaussian bounds we note that
by formula (9) of \cite{Du} for some constant $c>0$
$$
\Re \langle H_{\lambda \psi} f, f \rangle \ge  \langle (\Delta^2-\lambda^4 c) f, f \rangle.
$$
Now standard heat kernels theory argument  shows
$$
\|e^{-\lambda \psi}\exp(-t H)e^{\lambda \psi}\|_{2\to \infty } \le  C t^{-3/4} e^{c\lambda^4t}.
$$
This estimate  implies  off-diagonal Gaussian bounds ${\rm (GE_4)}$ by Davies' perturbation argument.
\end{proof}

Next   we establish  restriction type  estimates for spectral measure $dE_{\sqrt[4]{H}}(\lambda)$
 associated with the special higher order operator $H=\Delta^2+V$ with   potentials $V$ on $\mathbb{R}^3$.
Let   $H_0=\Delta^2$ be the self-adjoint extension operator on $L^2(\mathbb{R}^3)$.
Then we have $\sigma(H_0)=[0,\infty)$. For any $z\in \mathbb{C}\backslash \sigma(H_0)$,
  the resolvent
$$
R_0(z)=(H_0-z)^{-1}
$$
is well-defined on $L^2(\mathbb{R}^3)$.
We consider the boundary behavior of $R_0(z)$ as $z$
 approaches to some $\lambda>0$ since it is connected with the spectral measure by the limiting absorption principle:
\begin{eqnarray}\label{e5.7}
\frac{1}{2\pi i}\langle (R_0(\lambda+i0)-R_0(\lambda-i0))f,\ g\rangle=\langle dE_{H_0}(\lambda)f,\  g\rangle,
\ f, g\in \mathscr{S}(\mathbb{R}^3).
\end{eqnarray}
Let $\mu=\lambda+i\varepsilon$ where $ \lambda>0$  and $0<\varepsilon < \frac{\lambda}{10}$.
By elementary integration, it can be verfied  that

\begin{eqnarray*}
\int_{\mathbb{R}^3}\frac{e^{-i\xi x}}{|\xi|^2-\mu^2}\ d\xi=  \frac{e^{i\mu|x|}}{4\pi|x|}
\ \ \ \ {\rm and} \ \ \
\int_{\mathbb{R}^3}\frac{e^{-i\xi x}}{|\xi|^2+\mu^2}\ d\xi=  \frac{e^{-\mu|x|}}{4\pi|x|},
\end{eqnarray*}
which gives
 \begin{eqnarray} \label{e5.8}
\int_{\mathbb{R}^3}\frac{e^{-i\xi x}}{|\xi|^4-\mu^4}\ d\xi=  \frac{1}{2 \mu^2}\Big(\frac{e^{i\mu|x|}}{4\pi|x|}-\frac{e^{-\mu|x|}}{4\pi|x|}\Big{)}=
\frac{e^{i\mu|x|}}{4(1+i)\pi\mu}\bigg(\frac{1-e^{-(1-i)\mu|x|}}{(1-i)\mu|x|}\bigg).
\end{eqnarray}
Hence $K(\mu^4,x)$ - the Green kernel of $(\Delta^2-\mu^4)^{-1}$ is given by the formula
 \begin{eqnarray}\label{e5.9}
K(\mu^4,x)=\frac{e^{i\mu|x|}}{4(1+i)\pi\mu}\bigg(\frac{1-e^{-(1-i)\mu|x|}}{(1-i)\mu|x|}\bigg).
 \end{eqnarray}

\begin{proposition}\label{prop5.5}
Let
 $
R_0(\mu^4)=(H_0-\mu^4)^{-1}
 $ where  $H_0=\Delta^2$  acts on $ \mathbb{R}^3$.
If $\mu=\lambda+i\varepsilon$ with $ \lambda>0$  and $0< |\varepsilon |< \frac{\lambda}{10}$,
then for every $1\le p\le \frac{4}{3}$,
\begin{eqnarray}\label{e5.10}
\|R_0(\mu^4)\|_{p\to p'}\le C |\mu|^{3(\frac{1}{p}-\frac{1}{p'})-4}.
\end{eqnarray}
In particular, the following estimates of  incoming and outcoming operators are satisfied 
\begin{eqnarray}\label{e5.11}
\|R_0(\lambda^4-i0)\|_{p\to p'}=\|R_0(\lambda^4+i0)\|_{p\to
p'}\le C \lambda^{3(\frac{1}{p}-\frac{1}{p'})-4}.
\end{eqnarray}
\end{proposition}

\medskip

\begin{proof} Observe that by \eqref{e5.8} and \eqref{e5.9}, there exists a constant $C>0$ such that
  $|K(\mu^4,x)|\leq C|\mu|^{-1}.$
By  Young's inequality,
$$
\|R_0(\mu^4)f\|_\infty\leq C |\mu|^{-1}\|f\|_1.
$$
Now by the interpolation (see \cite{BL}),
  it suffices to verify \eqref{e5.10} for   $p={4/3}$, that is,
$
\|R_0(\mu^4)f\|_4\leq C |\mu|^{-5/2}\|f\|_{4/3}.
$
Observe  that
$|\mu|\sim \lambda$, by the scaling in $\lambda$   it reduces  to show that uniformly
\begin{eqnarray}
\label{e5.12}
\|R_0((1+i\epsilon)^4)f\|_4\leq C \|f\|_{4/3},  \ 0< \epsilon <1/10.
\end{eqnarray}

Now one can write
$$
R_0((1+i\epsilon)^4)={1\over 2(1+i\epsilon)^2}\big((-\Delta-(1+i\epsilon)^2)^{-1}-(-\Delta+(1+i\epsilon)^2)^{-1}\big).
$$
To estimate $L^{4/3}$ to $L^4$ norm of $(-\Delta-(1+i\epsilon)^2)^{-1}$,  one can use the argument from the proof of Theorem 2.3 in \cite{KRS},
see also Lemma 4   in \cite{GS}.
$L^{4/3}$ to $L^4$ norm estimates of $(-\Delta+(1+i\epsilon)^2)^{-1}$
are straightforward consequence of the standard Gaussian bounds.   \end{proof}

\begin{proposition}\label{prop5.6} Let
 $
R_0(\mu^4)=(H_0-\mu^4)^{-1}
 $ where  $H_0=\Delta^2$  acts on $ \mathbb{R}^3$ and  $\mu=\lambda+i\varepsilon$ where
 $ \lambda>0$  and $0\le |\varepsilon| < \frac{\lambda}{10}$. Suppose that
 $V\in L^1(\mathbb R^3)\cap L^2(\mathbb R^3)$.

 Then
  \begin{itemize}

\item[(i)]
  For $4\leq p \le \infty$, the operator map $\mu\mapsto R_0(\mu^4)V $ is continuous from the cone
 domain $\{\mu=\lambda\pm i\varepsilon, \lambda> 0 \ and\  0\le  \varepsilon < \frac{\lambda}{10}\} $
 to the space of bounded operators on $L^p(\mathbb R^3)$.

 \smallskip

 \item[(ii)]
 For $4\leq p \le \infty$
there exists a positive constant  $\lambda_0$ such that for all $\lambda\ge \lambda_0>0$
the operator $I+ R_0(\mu^4)V$ is invertible on $L^p(\mathbb R^3)$ and
$$\sup_{\lambda\ge \lambda_0}\|(I+ R_0(\mu^4)V)^{-1}\|_{p\to p}\le C.$$
\end{itemize}
\end{proposition}

\begin{proof}
Note that, for the case, we can write the $R_0(\mu^4)V$ into the following two parts:
$$R_0(\mu^4)V=((-\Delta-\mu^2)^{-1}V-(-\Delta+\mu^2)^{-1}V)/2\mu^2 . $$
Hence the proof  of (i)   follows from Lemmas 8 and   10 of \cite{GS}.

Next we prove (ii). Define the operator $M_V$ by the formula
$M_Vf(x)=V(x)f(x)$ and we note that if $V\in L^1(\mathbb R^3)\cap L^2(\mathbb R^3)$
then $ \|M_V\|_{p\to p'}< \infty $ for all $4\le p \le \infty $.
Now by Proposition~\ref{prop5.5} there exists a constant $\lambda_0>0$
$$\|R_0(\mu^4)V\|_{p\to p}\le \|R_0(\mu^4)\|_{p'\to p}\|M_V\|_{p\to p'} \le  {1\over 2}   $$
for all  $\lambda\ge \lambda_0$.  This proves (ii) and  concludes the proof of Proposition \ref{prop5.6}.
\end{proof}

\begin{proposition}\label{prop5.7} Suppose that
   $H=\Delta^2+V$ on ${{\mathbb R}^3}$ with a real-valued
    $V\in L^1(\mathbb R^3)\cap L^2(\mathbb R^3)$.

 Then
there exists a $\lambda_0>0$ such that
\begin{eqnarray*}
\|dE_{\sqrt[4]{H}}(\lambda)\|_{L^p\to L^{p'}}\leq C \lambda^{3({1\over p}-{1\over p'})-1}
\end{eqnarray*}
for all $\lambda\ge \lambda_0$ and $1\leq p\leq 4/3$.
\end{proposition}

\medskip

 \noindent
\begin{proof} Let   $\mu=\lambda+i\varepsilon$ where $ \lambda\ge\lambda_0>0$  and $0< |\varepsilon| < \frac{\lambda}{10}$.
 We denote  by $R(\mu^4)=(H-\mu^4)^{-1}$
  the resolvent of $H=\Delta^2+V$ on $L^2({\mathbb R}^3)$. Note that
\begin{eqnarray}
\label{e5.13}R(\mu^4)=(I+R_0(\mu^4)V)^{-1}R_0(\mu^4).
\end{eqnarray}

Hence it follows that from Propositions  \ref{prop5.5} and \ref{prop5.6}
$$\|R(\mu^4)f\|_{p'}\le \|(I+R_0(\mu^4)V)^{-1}\|_{p'\to p'} \|R_0(\mu^4)f\|_{p'}\le C{(\lambda_0, V)} \  |\mu|^{3({1\over p}-{1\over p'})-1}
\ \|f\|_{p}.$$
By the limit absorption principle the above estimates imply  Proposition~\ref{prop5.7}.
 \end{proof}

\medskip

Finally, we are now  able to state the following results describing spectral multipliers for the biharmonic
operators with some potential $V$ on ${\mathbb R^3}$

 \begin{theorem}\label{th5.8}
 Suppose that
   $H=\Delta^2+V$ on ${{\mathbb R}^3}$ with a  positive real-valued
    $0\leq V\in L^1(\mathbb R^3)\cap L^2(\mathbb R^3)$
   and that  $1\le p \leq 4/3$.
     Next assume that    $F$ is a bounded Borel  function such that  $\supp F\subseteq [1/4,4]$
 and $F \in W^{\alpha}_2({\mathbb R})$
for some $\alpha>3(1/p-1/2)$. Then for every $ p < r\le p'$, $F(t {H})$ is bounded on $L^r(X)$ for all $t>0$. In addition
\begin{equation}
\label{e5.14}
\sup_{t < 1/(16\lambda_0)}\|F(t {H})\|_{r\to r} \leq
C_r\|F\|_{W^{\alpha}_2},
\end{equation}
and
\begin{equation}
\label{e5.15}
\sup_{t \geq 1/(16\lambda_0)}\|F(t {H})\|_{r\to r} \leq
C_r\|F\|_{W^{\alpha}_{\infty}}
\end{equation}
for some constant $\lambda_0>0$    as in Proposition~\ref{prop5.7}.
\end{theorem}

\begin{proof}  The result is   a straightforward consequence of Propositions~\ref{prop5.4}, ~\ref{prop5.7}, ~\ref{prop4.4} and  Theorem~\ref{th4.5}.
 \end{proof}

\subsection{   Laplace type operators acting on fractals}
 Theorem~\ref{th4.2} can be applied to any operator which satisfies estimates ${\rm (GE_m)}$
and for which the ambient spaces satisfies the doubling condition. A compelling  class of such
operators is considered in the theory of  diffusion processes on fractals, see for example
\cite{BP, BS, Ki, Str1, Str2}.  One of the most well known space of this type is Sierpi\'nski gasket SG
see for example \cite{Ki,  Str2}. The Laplace operator on the Sierpi\'nski gasket SG (Neumann or Dirichlet)
satisfies Gaussian bound of order $m = \log 5/(\log 5-\log 3)$ and \eqref{e2.2}   holds with with the homogeneous
dimension given by  $n = \log 3/(\log 5 - \log 3) =2.1506601\ldots$,  see \cite{BP, Str1, Str2}.
Now application of  Theorem~\ref{th4.2} to this setting yields the following result.

 \begin{theorem} \label{th5.9} Suppose that $L$ is the Laplacian on the
 Sierpi\'nski gasket.  Let $1\leq p<2$.
   Then for any  bounded Borel
function $F$ such that
$\sup_{t>0}\|\eta\, \delta_tF\|_{W^{\alpha}_{\infty}}<\infty $ for    some
$\alpha> n(1/p-1/2),$ the operator
$F(L)$ is bounded on $L^r(X)$ for all $p<r<p'$.
In addition,
\begin{eqnarray*}
   \|F(L)  \|_{r\to r}\leq    C_\alpha\Big(\sup_{t>0}\|\eta\, \delta_tF\|_{W^{\alpha}_{\infty}}
   + |F(0)|\Big),
\end{eqnarray*}
where  $n = \log 3/(\log 5 - \log 3)
 $ and $m =
 \log 5/(\log 5-\log 3)$.
 \end{theorem}

\begin{proof} The result is direct consequence of   Theorem~\ref{th4.2}.
 \end{proof}

We do not know however if  Theorem  \ref{th5.9} is sharp. The case $p=1$ of this result
is discussed in details in \cite{DOS}. Theorem  \ref{th5.9} can be extend to include broader class of fractals.
One simple  class of possible  generalization can be given by products of any number of copies of Sierpi\'nski gaskets.

\bigskip

\noindent
{\bf Acknowledgements:} A. Sikora was partly  supported by
Australian Research Council  Discovery Grant DP 110102488. L.  Yan was supported by
 NNSF of China (Grant No.  10925106),
 Guangdong Province Key Laboratory of Computational Science
 and Grant for Senior Scholars from the Association of Colleges and Universities of Guangdong.
  Part of this work was done while
 L. Yan was visiting Macquarie University. His visit was partly supported
 by ARC Discovery Grant DP 110102488. He wishes
 to thank X.T. Duong for the invitation.
 X. Yao was  supported by NCET-10-0431 and the Special Fund for Basic Scientific Research of Central Colleges
 (No. CCNU12C01001).

\end{document}